\documentclass[11pt,reqno]{article}
\usepackage[utf8]{inputenc}
\usepackage{booktabs}
\usepackage{array}
\usepackage{paralist}
\usepackage{verbatim}
\usepackage{amsthm}
\usepackage[table,xcdraw]{xcolor}
\usepackage[titletoc,toc,title]{appendix}
\usepackage{latexsym, amsmath, amssymb, a4, epsfig, color}
\usepackage{blindtext}
\usepackage{graphicx}
\usepackage{subcaption}
\usepackage[utf8]{inputenc}
\usepackage[export]{adjustbox}
\usepackage{wrapfig}
\usepackage{apptools}
\usepackage{afterpage}
\usepackage{floatrow}
\usepackage{amssymb}
\usepackage{amsmath}
\usepackage[normalem]{ulem}
\usepackage{enumitem}
\usepackage{appendix}

\usepackage{multirow}
\usepackage{array}
\newcolumntype{C}[1]{>{\centering\arraybackslash$}p{#1}<{$}}

\usepackage[french,english]{babel}

\AtAppendix{\counterwithin{definition}{subsection}}
\AtAppendix{\counterwithin{theorem}{subsection}}

\graphicspath{}

\newtheorem{theorem}{Theorem}[section]
\newtheorem{cor}[theorem]{Corollary}
\newtheorem{lemma}[theorem]{Lemma}
\newtheorem{prop}[theorem]{Proposition}
\newtheorem{definition}{Definition}

\newtheorem{example}{Example}
\newtheorem{remark}{Remark}

\newcommand{\Ga}{\alpha}
\newcommand{\Gb}{\beta}

\newcommand{\Gs}{\sigma}

\newcommand{\GO}{\Omega}

\newcommand{\FF}{\mathbb{F}}
\newcommand{\RR}{\mathbb{R}}
\newcommand{\CC}{\mathbb{C}}
\newcommand{\NN}{\mathbb{N}}
\newcommand{\ZZ}{\mathbb{Z}}

\newcommand{\Om}{\Omega}
\newcommand{\ds}{\displaystyle}
\newcommand{\pf}{\noindent {\sl Proof}. \ }
\newcommand{\p}{\partial}
\newcommand{\pd}[2]{\frac {\p #1}{\p #2}}
\newcommand{\eqnref}[1]{(\ref {#1})}
\renewcommand{\qed}{\hfill $\Box$ \medskip}
\newcommand{\beq}{\begin{equation}}
\newcommand{\eeq}{\end{equation}}
\newcommand{\be}{\begin{equation*}}
\newcommand{\ee}{\end{equation*}}

\newcommand{\Kcal}{\mathcal{K}}
\newcommand{\Scal}{\mathcal{S}}

\newcommand{\la}{\langle}
\newcommand{\ra}{\rangle}

\numberwithin{equation}{section}
\numberwithin{figure}{section}
\numberwithin{table}{section}
\numberwithin{example}{section}
\numberwithin{remark}{section}
\numberwithin{definition}{section}

\arraycolsep=1.3pt\def\arraystretch{1.3}

\begin{document}

\newcommand{\TheTitle}{Geometric multipole expansion and its application to semi-neutral inclusions of general shape}
\newcommand{\TheAuthors}{D. Choi, J. Kim and M. Lim}

\title{{\TheTitle}\thanks{This work was supported by the National Research Foundation of Korea(NRF) grant funded by the Korea government(MSIT) (NRF-2016R1A2B4014530 and NRF-2019R1F1A1062782).}}
\author{Doosung Choi\thanks{\footnotesize Department of Mathematical Sciences, Korea Advanced Institute of Science and Technology, Daejeon 34141, Korea ({7john@kaist.ac.kr}, {kjb2474@kaist.ac.kr},  {mklim@kaist.ac.kr}).}\and Junbeom Kim\footnotemark[2] \and Mikyoung Lim\footnotemark[2]}
\date{\today}
\maketitle
\begin{abstract}
This paper presents a new concept of geometric multipole expansion for the conductivity or anti-plane elasticity problem in two dimensions by using the Faber polynomials. As an application, we construct semi-neutral inclusions of general shape that show relatively negligible field perturbations for low-order polynomial loadings. These inclusions are of the multilayer structure whose material parameters are determined such that some coefficients of geometric multipole expansion vanish.
\end{abstract}

%\begin{otherlanguage}{french}
%\begin{abstract}
%Cet article pr\'{e}sente un nouveau concept d'expansion multipolaire g\'{e}om\'{e}trique pour le probl\`{e}me de conductivit\'{e} ou d'\'{e}lasticit\'{e} antiplane en deux dimensions en utilisant les polyn\^{o}mes de Faber. En tant qu'application, nous construisons des inclusions semi-neutres de forme g\'{e}n\'{e}rale qui montrent des perturbations de champ relativement n\'{e}gligeables pour les chargements polynomiaux d'ordre inf\'{e}rieur. Ces inclusions sont de la structure multicouche o\`{u} les param\`{e}tres du mat\'{e}riau sont d\'{e}termin\'{e}s de telle sorte que certains coefficients de l'expansion multipolaire g\'{e}om\'{e}trique disparaissent.
%\end{abstract}
%\end{otherlanguage}


\section{Introduction}

We consider the field perturbation due to the presence of an elastic or electrical inclusion in a homogeneous background $\RR^2$. An elastic or electrical inclusion with different material parameters from that of the background induces a perturbation on the applied background field. 
For this conductivity transmission problem, one can find the solution using a single-layer potential ansatz, where the density function involves the so-called Neumann--Poincar\'{e} operator. This boundary integral formulation provides us the classical multipole expansion of the field perturbation, whose coefficients are the so-called generalized polarization tensors (GPTs) \cite{Ammari:2013:MSM, Ammari:2007:PMT}.
The classical multipole expansion holds in a far-field region, but, in general, it does not hold near the boundary of the inclusion. Consequently, the classical multipole expansion cannot be employed to find the solution to the transmission problem; it provides a solution only when the inclusion is a circular or spherical domain.

In this paper, to overcome the limitation of the classical multipole expansion, we propose a geometric multipole expansion applicable to solving the conductivity transmission problem with an inclusion of general shape.
We assume that the inclusion is either a simply connected or multilayered domain whose layers are enclosed by images of concentric circles via the exterior conformal mapping of the core. We refer the reader to Figure \ref{fig:multicoated} in Section \ref{sec:FPT_reducing} for the geometry of such a multilayered domain.

Complex analysis techniques have been used to study various inclusion problems in two dimensions \cite{Ammari:2018:MNF, Bonnetier:2012:PBG, Choi:2018:CEP, Jung:2021:SEL}. 
For any simply connected region, there uniquely exists an exterior conformal mapping. The Faber polynomials are then defined depending on the exterior conformal mapping \cite{Faber:1903:PE}, where they form a basis for analytic functions in the region \cite{Duren:1983:UF}. Recently, Jung and Lim \cite{Jung:2021:SEL} obtained series expansions of the layer potential operators based on geometric function theory; we refer the reader to \cite{Choi:2021:EEC, Jung:2020:DEE} for their applications.

As the main results, we introduce the geometric multipole expansion by using the Faber polynomials. Unlike the classical multipole expansion, this expansion holds in the whole exterior of the inclusion, and, consequently, one can solve the transmission problem using this expansion. We then define the Faber polynomial polarization tensors (FPTs) that are coefficients of the geometric multipole expansion. The FPTs coincide with the GPTs for the disk case, and, in general, are linear combinations of the GPTs with weights determined by the Faber polynomials. We provide a matrix expression for the FPTs in terms of the material parameter and the exterior conformal mapping of the inclusion. It is worth remarking that the FPTs was successfully applied for analytical shape recovery of a conductivity inclusion \cite{Choi:2020:ASR}.

Coated disks and spheres are well-known examples of neutral inclusions, that is, structures not disturbing the applied uniform field \cite{Hashin:1962:EMH, Hashin:1985:LIE, Hashin:1962:VAT, Jimenez:2013:NNI}. 
Appropriately coated ellipses and ellipsoids, possibly with the anisotropic conductivity, are neutral to all uniform exterior fields \cite{Grabovsky:1995:MME1, Kerker:1975:IB, Milton:2002:TC, Sihvola:1999:EMF, Sihvola:1997:DPI}, and they are the only shapes for which coated inclusions have the uniform field property \cite{Kang:2014:CIF, Kang:2016:OBV, Milton:2001:NCI}. The idea of neutral inclusion has been widely studied for the invisible cloaking using metamaterials. For instance, Zhou et al.~designed multi-coated spheres that are invisible to acoustic, elastic, and electromagnetic waves \cite{Zhou:2006:DEW, Zhou:2007:AWT, Zhou:2008:EWT}. After then, Landy and Smith \cite{Landy:2013:FPU} experimentally characterized the neutral inclusions with microwaves. For the case of Maxwell's equations, Al\`{u} and Engheta \cite{Alu:2005:ATP} and Ammari et al.~\cite{Ammari:2013:ENCm} constructed multi-coated neutral inclusions.
 The GPT-vanishing structures are concentric disks or balls whose values of the GPTs are negligible for leading orders \cite{Ammari:2013:ENC1, Wang:2013:MNC}.  One can interpret them as multi-coated neutral inclusions.
It is worth remarking that inclusions of general shapes that cancel the first-order GPTs were constructed \cite{Feng:2017:CGV, Kang:2019:CWN}.

As an application of the FPTs, we construct multi-coated inclusions, for a given core of general shape, that show relatively negligible field perturbations for low-order polynomial loadings. We call such a structure a {\it semi-neutral inclusion}. 
The coating layers of this inclusion are images of concentric circles via the exterior conformal mapping of the core. 
The FPTs can be divided into two groups $\FF^{(1)}$ and $\FF^{(2)}$ (see Theorem \ref{F1F2ab} in Section \ref{FPTGME}); the first mainly depends on the shape of the inclusion, and the second more depends on the material parameters.
For concentric disks, $\FF^{(1)}=0$ due to the symmetry of the shape \cite{Ammari:2013:ENC1}. Hence, the GPT-vanishing structures obtained in \cite{Ammari:2013:ENC1} are in fact the $\FF^{(2)}$-vanishing structures of concentric multi-coated disks. In general, $\FF^{(2)}$ shows a larger magnitude compared to $\FF^{(1)}$, and $\FF^{(2)}$ significantly contributes to the field perturbation. We numerically find semi-neural inclusions such that $\FF^{(2)}$-terms of leading orders vanish by a simple optimization procedure.

The paper is organized as follows. In Section \ref{sec:pre}, we review the boundary integral formulation for the transmission problem and outline the series expansions of the layer potential operators. In Section \ref{sec:cMPE}, we define the GPTs and the classical multipole expansion, and then in Section \ref{FPTGME}, we expand these concepts to inclusions of general shape by the FPTs and the geometric multipole expansion. Section \ref{sec:FPT_reducing} is to analytically compute matrix formulas of FPTs for a multi-coated structure. By using this formula, we then construct semi-neutral inclusions and show numerical examples in Section \ref{sec:numerical}.

%%%%%%%%%%%%%%%%%%%%%%%%%%%%%%%%%

\section{Preliminary}\label{sec:pre}

\subsection{Layer potential technique for the conductivity transmission problem}

Let $D$ be a bounded and simply connected domain in $\RR^2$ with Lipschitz boundary. We assume that $D$ has the constant conductivity $\sigma_0>0$ and is embedded in the background with the constant conductivity $\sigma_m$. For simplicity, we assume $\sigma_m=1$. We consider the resulting conductivity (or anti-plane elasticity) transmission problem in two dimensions:
\beq\label{cond_eqn0}
\begin{cases}
	\ds\nabla\cdot\sigma\nabla u=0\quad&\mbox{in }\RR^2, \\
	\ds u(x) - H(x)  =O({|x|^{-1}})\quad&\mbox{as } |x| \to \infty
\end{cases}
\end{equation}
with the conductivity distribution given by 
$\sigma = \sigma_0\, \chi(D)+\chi(\RR^2\setminus\overline{D})$ and an entire harmonic function $H$.
Here, $\chi$ indicates the characteristic function. It holds that 
\beq\label{u:trans}
u\big|^+=u\big|^- ,\quad \pd{u}{\nu}\Big|^+= \sigma_0 \pd{u}{\nu}\Big|^-\quad\mbox{on }\p D.
\eeq
The symbols $+$ and $-$ indicate the limits from the exterior and interior of $\p D$, respectively.

For $\varphi\in L^2(\p D)$, we define
\begin{align*}
\ds&\Scal_{\p D}[\varphi](x)=\int_{\p D}\Gamma(x-\tilde{x})\varphi(\tilde{x})\, d\sigma(\tilde{x}),\quad x\in\RR^d,\\
\ds &\Kcal_{\p D}^*[\varphi](x)=p.v.\,\frac{1}{2\pi}\int_{\partial D}\frac{\left\la x-\tilde{x},\nu_x\right\ra}{|x-\tilde{x}|^2}\varphi(\tilde{x})\, d\sigma(\tilde{x}),\quad x\in\p D,
\end{align*}
where  $\Gamma(x)$ is the fundamental solution to the Laplacian, i.e., 
$\Gamma(x)=\frac{1}{2\pi}\ln|x|$,  $p.v.$ stands for the Cauchy principal value, and $\nu_x$ is the outward unit normal vector to $\p D$ at $x$. We call $\Scal_{\p D}[\varphi]$ and $\Kcal^*_{\p D}$ the single-layer potential
and the Neumann--Poincar\'{e} (NP) operator, respectively.  On $\p D$, the following jump relation holds:
\beq \label{eqn:Kstarjump}
\begin{aligned}
\Scal_{\p D}[\varphi]\big|^{+}&=\Scal_{\p D}[\varphi]\big|^{-},\\
\frac{\partial}{\partial\nu}\Scal_{\p D}[\varphi]\Big|^{\pm}&=\left(\pm\frac{1}{2}I+\Kcal^*_{\p D}\right)[\varphi].
\end{aligned}
\eeq
The $L^2$ adjoint of $\Kcal^*_{\p D}$ is 
\be
\ds \Kcal_{\p D}[\varphi](x)=p.v.\,\frac{1}{2\pi}\int_{\partial D}\frac{\left\la \tilde{x}-{x},\nu_{\tilde{x}}\right\ra}{|x-\tilde{x}|^2}\varphi(\tilde{x})\, d\sigma(\tilde{x}),\quad x\in\p D.
\ee
We also call $\Kcal_{\p D}$ the NP operator by an abuse of terminology. 
The operator $\Kcal^*_{\p D}$ can be extended to act on the Sobolev space $H^{-1/2}(\p D)$ by using its $L^2$ pairing with $\Kcal_{\p D}$. 
We identify $x=(x_1,x_2)$ in $\RR^2$ with the complex variable $z=x_1+ix_2$ in $\CC$. 
We denote
$\Scal_{\p D}[\varphi](z):=\Scal_{\p D}[\varphi](x)$ and similarly for the NP operators.

From \eqnref{u:trans} and \eqnref{eqn:Kstarjump}, the solution $u$ to \eqnref{cond_eqn0} admits the single-layer potential ansatz:
\beq\label{umh1}
u(x)=H(x)+\Scal_{\p D}[\varphi](x),\quad x\in\RR^2,
\eeq
where
\beq\label{umh2}
\varphi=(\lambda I-\Kcal_{\p D}^*)^{-1}\left[\nu\cdot \nabla H\right]\quad\mbox{with }\lambda = \frac{\sigma_0+1}{2(\sigma_0-1)}.
\eeq
The operator $\lambda I-\Kcal_{\p D}^*$ is invertible on $L^2_0(\p D)$ (or $H_0^{-1/2}(\p D)$) for $|\lambda|\geq 1/2$ \cite{Escauriaza:1992:RTW, Kellogg:1929:FPT, Verchota:1984:LPR} (see also \cite{Escauriaza:1992:RTW, Escauriaza:1993:RPS} for the stability results). We refer the reader to \cite{Helsing:2013:SIE, Helsing:2017:CSN} for the numerical computation with high precision and to \cite{Ammari:2007:PMT} and references therein for more properties of the NP operators and their applications.

The operator $\Kcal^*_{\p D}$ is symmetric in $L^2(\p D)$ only for a disk or a ball \cite{Lim:2001:SBI}. 
However, the NP operators can be symmetrized using Plemelj's symmetrization principle \cite{Khavinson:2007:PVP}:
\be
\Scal_{\p D}\Kcal_{\p D}^*=\Kcal_{\p D}\Scal_{\p D}.
\ee
We denote by $H^{-1/2}_0(\p D)$ the subspace of functions $u$ contained in $H^{-1/2}(\p D)$ such that $$\la u, 1\ra_{-\frac{1}{2},\frac{1}{2}}=0,$$ where $\la\cdot,\cdot\ra_{-\frac{1}{2},\frac{1}{2}}$ is the duality pairing between the Sobolev spaces $H^{-1/2}(\p D)$ and $H^{1/2}(\p D)$.
The operator $\Kcal_{\p D}^*$ is self-adjoint in $H^{1/2}_0(\p D)$ equipped with a new inner product that involves the single-layer potential \cite{Ando:2016:APR, Kang:2015:LPA, Khavinson:2007:PVP}.

\subsection{Exterior conformal mapping and associated orthogonal coordinates}\label{subsec:coord}
From the Riemann mapping theorem, there uniquely exist a real number $\gamma>0$ and a complex function $\Psi(w)$ that conformally maps the region $\{w\in\CC:|w|> \gamma\}$ onto $\mathbb{C}\setminus\overline{D}$ and satisfies $\Psi(\infty)=\infty$ and $\Psi'(\infty)=1$. We set $\rho_0=\ln \gamma$. The function $\Psi(w)$ admits the following Laurent series expansion:
\beq\label{eqn:extmapping} 
\Psi(w)=w+{a}_0+\frac{{a}_1}{w}+\frac{{a}_2}{w^2}+\cdots=w+\sum_{k=0}^{\infty}a_kw^{-k}
\eeq 
for some complex coefficients $a_n$. We call $\gamma$ the conformal radius of $D$. 
From the well-known Bieberbach conjecture, it holds that 
\beq \label{ineq:a1}
|a_1|<\gamma^2
\eeq
assuming that the area of $D$ is positive. 
From the Caratheodory extension theorem \cite{Caratheodory:1913:GBR}, $\Psi(\rho,\theta)$ extends to the boundary of $D$ as a homeomorphism. The conformal mapping $\Psi$ defines an orthogonal curvilinear coordinate system $(\rho,\theta)\in[\gamma,\infty)\times[0,2\pi)$ for each $z$ in $\CC\setminus D$ via the relation
\be
z=\Psi(\rho,\theta):=\Psi(e^{\rho+\mathrm{i}\theta}).
\ee
The scale factors with respect to $\rho$ and $\theta$ coincide with each other. We denote them by
\be
h(\rho, \theta) := \left|\frac{\partial \Psi} {\partial \rho}\right| = \left|\frac{\partial \Psi} {\partial \theta}\right|.
\ee	
The length element on $\p D$ is given by $d\sigma(z)=h(\rho_0,\theta)d\theta$, and for a function $v(z)=(v\circ\Psi)(\rho,\theta)$ defined in the exterior of $D$, it holds that
\be
\frac{\partial v}{\partial \nu}\Big|_{\p D}^{+}(z)=\frac{1}{h(\rho_0,\theta)}\frac{\partial }{\partial \rho}v\left(\Psi(\rho,\theta)\right)\Big|_{\rho\rightarrow\rho_0^+}.
\ee

If we further assume that $D$ is a $C^{1,\alpha}$ domain for some $0<\alpha<1$, then, by the Kellogg--Warschawski theorem \cite{Pommerenke:1992:BBC}, $\Psi'$ can be continuously extended to the boundary.

As a univalent function, $\Psi$ defines the so-called Faber polynomials $\{F_m(z)\}_{m=1}^\infty$, which were first introduced by G. Faber \cite{Faber:1903:PE} and have been extensively studied in various areas. They are defined by the relation
\beq\label{def:Farber}
\frac{w\Psi'(w)}{\Psi(w)-z}=\sum_{m=0}^\infty \frac{F_m(z)}{w^{m}},\quad z\in{\overline{D}},\ |w|>\gamma.
\eeq
This provides explicit expressions for $F_m$ in terms of $a_n$.
 For example, 
$F_0(z)=1,\ F_1(z)=z-a_0,\ F_2(z)=z^2-2a_0 z+a_0^2-2a_1.$
The Faber polynomials form a basis for complex analytic functions in $D$ \cite{Duren:1983:UF}.
An essential property of the Faber polynomials is that $F_m(\Psi(w))$ is the addition of $w^m$ and negative order terms. In other words,
 \begin{equation} \label{eqn:Faberdefinition}
F_m(\Psi(w))=w^m+\sum_{n=1}^{\infty}c_{mn}{w^{-n}},\quad |w|>\gamma,
\end{equation}
where the coefficients $c_{mn}$ are called the Grunsky coefficients. It holds the Grunsky identity: $$n c_{mn}=m c_{nm}$$ for all $m,n\in\NN$. 
 One can obtain the Grunsky coefficients from the exterior conformal mapping by the recursive formula:
\beq\label{grunskyformula}
c_{m(n+1)} = c_{(m+1)n} - a_{m+n} + \sum_{s=1}^{m-1} a_{m-s}c_{sn} - \sum_{s=1}^{n-1} a_{n-s}c_{ms}, \quad  m,n\ge 1
\eeq
with initial values $c_{1n} = a_n$ and $c_{n1} = na_n$, $n\ge1$.

The complex logarithm admits the following expansion \cite{Duren:1983:UF, Faber:1903:PE, Jung:2021:SEL}: for ${z}=\Psi(w)\in\CC\setminus\overline{D}$ and $\tilde{z}\in D$,
\beq\label{log:Faber}
\log({z}-\tilde{z})=\log w-\sum_{m=1}^\infty \frac{1}{m}F_m(\tilde{z})w^{-m}\eeq
with a proper branch cut.
The expansion \eqnref{log:Faber} sheds new light to understand the solution of the transmission problem \eqnref{cond_eqn0} and the NP operator \cite{Jung:2020:DEE, Jung:2021:SEL}.

The Grunsky coefficients satisfy the so-called {\it strong Grunsky inequalities} \cite{Duren:1983:UF, Grunsky:1939:KSA}: let $N$ be a positive integer and $\lambda_1,\lambda_2,\dots,\lambda_N$ be complex numbers that are not all zero, then we have
\beq\label{inequal:strong}
\sum_{n=1}^\infty n\left|\sum_{k=1}^N\frac{c_{kn}}{\gamma^{k+n}}\lambda_k \right|^2\leq\sum_{n=1}^N n\left|\lambda_n \right|^2,
\eeq
where the equality holds if and only if $D$ is of measure zero.
We also have the so-called {\it weak Grunsky inequality}:
\beq\label{inequal:weak}
\left|\sum_{s=1}^N\sum_{k=1}^N s\frac{c_{ks}}{\gamma^{k+s}}\lambda_k \lambda_s\right|\leq\sum_{s=1}^N s\left|\lambda_s \right|^2.
\eeq
For fixed $m$, plugging $\lambda_k=\frac{1}{\sqrt{m}}\delta_{mk}$ into \eqnref{inequal:strong}, we have 
\be
\sum_{n=1}^\infty \left|\sqrt{\frac{n}{m}}\frac{c_{mn}}{\gamma^{m+n}}\right|^2\leq 1.
\ee
In particular, $\left|\frac{c_{mm}}{\gamma^{2m}}\right|\leq 1$. 
For fixed $m$ and $n$ ($m\neq n$), letting $\lambda_k=\frac{1}{m}\delta_{km}+\frac{1}{n}\delta_{kn}$, we have from \eqnref{inequal:weak} that
$$\left|\frac{1}{m}\frac{c_{mm}}{\gamma^{2m}}+\frac{1}{m}\frac{c_{mn}}{\gamma^{m+n}}+\frac{1}{n}\frac{c_{nm}}{\gamma^{n+m}}+\frac{1}{n}\frac{c_{nn}}{\gamma^{2n}}\right|\leq\frac{1}{m}+\frac{1}{n}\leq 2,$$
and thus,
\beq \notag
\left|c_{mn}\right|\leq 2 m \gamma^{m+n}.
\eeq
It then follows from \eqnref{eqn:Faberdefinition} that for $|w|>\gamma$,
\beq\label{Faber:bound}
\left|F_m(\Psi(w))\right|
\leq |w^m|+\sum_{n=1}^{\infty}\left|c_{mn}{w^{-n}}\right|
\leq |w|^m+2m\gamma^m\frac{\gamma}{|w|-\gamma}. 
\eeq

Let $v$ be a complex analytic function in $D_R:=D\cup\{\Psi(w):\gamma\leq|w|<R\}$ for some $R>\gamma$. Fix any $r\in(\gamma,R)$. Then, \eqnref{def:Farber} holds also for $z\in\overline{D_r}$ and $|w|>r$. By applying the Cauchy integral formula to $v$ and by applying \eqnref{def:Farber}, it follows that
\beq\label{exp_anal_function_Faber}
v(z)=\sum_{m=0}^{\infty}b_m F_m(z)\quad\mbox{in }\overline{D_r}
\eeq
with
\be
b_m = \frac{1}{2\pi i}\int_{|w|=R'}\frac{v(\Psi(w))}{w^{m+1}}\, dw\quad\mbox{for any } R'\in(r,R).
\ee
Here, $b_m$ is independent of choice of $R'$ and $|b_m|\leq M (R')^{-m}$ for some constant $M$. From \eqnref{Faber:bound} and the maximum principle for complex analytic functions, \eqnref{exp_anal_function_Faber} uniformly and absolutely converges for $z\in\overline{D_r}$. Furthermore, \eqnref{eqn:Faberdefinition} and \eqnref{Faber:bound} imply that 
\beq\label{expan:v:Faber}
v(z)=\sum_{m=0}^{\infty}b_m\left(w^m+\sum_{n=1}^{\infty}c_{mn}{w^{-n}}\right)
\eeq
converges uniformly and absolutely for $z\in\left\{\Psi(w):r_1\leq|w|\leq r_2\right\}$ for any $\gamma<r_1<r_2<R$.
In particular, we can change the order of summation in \eqnref{expan:v:Faber}.

\subsection{Series expansions of layer potential operators using Faber polynomials}\label{sec:gMPE}
In this subsection, we review the series expansions of the single-layer potential and the NP operator that were developed in \cite{Jung:2021:SEL} using the exterior conformal mapping and the Faber polynomials associated with the inclusion.

We set the density basis functions: for $z=\Psi(\rho,\theta)\in\p D$,
\begin{align*}
&\eta_0(z)=1,\quad \zeta_0(z)=\frac{1}{h(\rho_0,\theta)},\\
&\eta_{m}(z)=|m|^{-\frac 1 2}e^{i m \theta},\quad
\zeta_{m}(z)=|m|^{\frac 1 2}\frac{e^{ i m \theta}}{h(\rho_0,\theta)}\quad\mbox{for }m\in\ZZ\setminus\{0\}.
\end{align*}
If $D$ has a $C^{1,\alpha}$ boundary, then $\zeta_m$ (resp. $\eta_m$) form a basis of $H^{-1/2}(\p D)$ (resp. $H^{1/2}(\p D)$) \cite{Jung:2021:SEL}. Furthermore, $\zeta_m$ and $\eta_m$ jointly form a bi-orthogonal system for the pair of spaces $H^{-1/2}(\p D)$ and $H^{1/2}(\p D)$. 
In particular, it holds that for any $f\in H^{1/2}(\p D)$ and $g\in H^{-1/2}(\p D)$,
\begin{align}\label{eta_exp:f}
f=&\sum_{m\in\ZZ} f_m\eta_m\quad\mbox{with}\quad f_{m}=\frac{1}{2\pi}\int_{\p D}f\, \overline{\zeta_m}\, d\sigma,\\ \notag
g=&\sum_{m\in\ZZ} g_m\zeta_m\quad\mbox{with}\quad g_{m}=\frac{1}{2\pi}\int_{\p D} g\, \overline{\eta_m}\, d\sigma.
\end{align}

\begin{theorem}[\cite{Jung:2021:SEL}] \label{theorem:seriesexpan}
Let $D$ be a bounded and simply connected domain in $\RR^2$ with $C^{1,\alpha}$ boundary for some $\alpha>0$.  For $z=\Psi(w)\in\CC\setminus\overline{D}$ with $w=e^{\rho+i\theta}$, the single-layer potential satisfies 
\beq\notag
\Scal_{\p\Om}[\zeta_0](z)=
\begin{cases}
\ln \gamma \quad &\mbox{in }\overline{D},\\
\ln|w|\quad&\mbox{in }\CC\setminus\overline{D},
\end{cases}
\eeq 
and, for $m\in\NN$,
\beq\label{Scal:expan}
\Scal_{\p D}[\zeta_m](z)=
\begin{cases}
\ds-\frac{1}{2\sqrt{m}\gamma^m}F_m(z)\quad&\text{in }\overline{D},\\
\ds-\frac{1}{2\sqrt{m}\gamma^m}\left(\sum_{n=1}^{\infty}c_{mn}e^{-n(\rho+i\theta)}+\gamma^{2m}e^{m(-\rho+i\theta)}\right)\quad &\text{in } \CC\setminus\overline{D}.
\end{cases}
\eeq
The series converges uniformly for all $(\rho,\theta)$ such that $\rho\geq\rho_1$ for any fixed $\rho_1>\rho_0$. For the density functions with negative index, it holds that $\Scal_{\p D}[\zeta_{-m}](z)=\overline{\Scal_{\p D}[\zeta_{m}](z)}.$

Furthermore, the NP operators satisfy
$$\Kcal^*_{\p D}\left[\zeta_0\right]=\frac{1}{2}\zeta_0,\quad \Kcal_{\p D}\left[1\right]=\frac{1}{2}$$
and
\begin{align}\label{NP_series1}
&\Kcal^*_{\p D}\left[{\zeta_m}\right]=\frac{1}{2}\sum_{n=1}^{\infty}\frac{\sqrt{n}}{\sqrt{m}}\frac{c_{mn}}{\gamma^{m+n}}\, {\zeta}_{-n},\quad
\Kcal^*_{\p D}\left[\zeta_{-m}\right]=\frac{1}{2}\sum_{n=1}^{\infty}\frac{\sqrt{n}}{\sqrt{m}}\frac{\overline{c_{mn}}}{\gamma^{m+n}}\, \zeta_{n},\\
\label{NP_series2}
&\Kcal_{\p D}\left[\eta_m\right]=\frac{1}{2}\sum_{n=1}^{\infty}\frac{\sqrt{n}}{\sqrt{m}}\frac{c_{mn}}{\gamma^{m+n}}\, {\eta_{-n}},
\quad
\Kcal_{\p D} \left[\eta_{-m}\right]=\frac{1}{2}\sum_{n=1}^{\infty}\frac{\sqrt{n}}{\sqrt{m}}\frac{\overline{c_{mn}}}{\gamma^{m+n}}\, \eta_{n}.
\end{align}
The infinite series in \eqnref{NP_series1} converge in $H^{-1/2}(\p D)$, and those in \eqnref{NP_series2} converge in $H^{1/2}(\p D)$. 
\end{theorem}

We have from \eqnref{eta_exp:f} that
$$\Scal_{\p D}[\zeta_m]=\sum_{n\in\ZZ} a_{mn}\eta_n$$
with
\begin{align*}
a_{mn}&=\frac{1}{2\pi}\int_{\p D}\Scal_{\p D}[\zeta_m]\, \overline{\zeta_n}\, d\sigma\\
&=\frac{1}{2\pi}\lim_{t\rightarrow 0^+}\int_{\p D}\Scal_{\p D}[\zeta_m](\rho_0+t,\theta)\, \overline{\zeta_n}\, d\sigma,
\end{align*}
where the second equality follows from the continuity of the single-layer potential.
From \eqnref{Scal:expan} and \eqnref{NP_series2}, we then have the following relation in $H^{1/2}(\p D)$ sense:
\begin{align}\label{Scal:zeta_m}
\Scal_{\p D}\left[\zeta_m\right]&=-\left(\frac{1}{2}I+\Kcal_{\p D}\right)\left[\eta_m\right]\quad\mbox{on }\p D.
\end{align}

\section{Classical and geometric multipole expansions}\label{sec:cMPE}

\subsection{Classical multipole expansion and CGPTs}
For a multi-index $\alpha=(\alpha_1,\alpha_2)\in \NN\times\NN$, we set $x^{\alpha}=x_1^{\alpha_1}x_2^{\alpha_2}$ and $|\alpha|=\alpha_1+\alpha_2$.
Applying the Taylor series method, the integral formula \eqnref{umh1} leads to the multipole expansion \cite{Ammari:2007:PMT}:
 \beq\label{CP2}
 u(x) = H(x)+\sum_{|\Ga|,|\Gb|=1}^\infty\frac{(-1)^{|\Gb|}}{\Ga!\Gb!}\p^\Gb\Gamma(x)M_{\Ga\Gb}(D,\sigma_0)\p^\Ga H(0), \quad |x|\gg1.
 \eeq
with
\beq
\label{gpt}
M_{\alpha\beta}(D,\sigma_0) = \int_{\p D} y^\beta \left(\lambda I - \Kcal^*_{\p\GO}\right)^{-1}\left[\pd{x^\alpha}{\nu}\right](y) \, d\Gs (y).
\end{equation}
The terms $M_{\Ga\Gb}(D,\sigma_0)$ are the so-called generalized polarization tensors (GPTs) corresponding to the inclusion $D$ with the conductivity $\sigma_0$.

%, 
Now, we identify $x=(x_1,x_2)$ in $\RR^2$ with $z=x_1+ix_2$ in $\CC$ and define the GPTs in complex form:
\begin{definition}[\cite{Ammari:2013:MSM}]
Let $\lambda=\frac{\sigma_0+1}{2(\sigma_0-1)}$, and, for each $n\in\NN$, $P_n(z)=z^n$. For $m,n\in\NN$, we define
\begin{align*}
\NN_{mn}^{(1)}(D,\sigma_0)&=\int_{\p D} P_n(z) \left(\lambda I-\Kcal^*_{\p D}\right)^{-1}\left[\pd{ P_m }{\nu} \right](z) \,d\sigma(z),\\
\NN_{mn}^{(2)}(D,\sigma_0)&=\int_{\p D} P_n(z) \left(\lambda I-\Kcal^*_{\p D}\right)^{-1}\left[\pd{\overline{P_m}}{\nu}\right](z) \,d\sigma(z).
\end{align*}
We call $\NN_{mn}^{(1)}$ and $\NN_{mn}^{(2)}$ the complex generalized polarization tensors (CGPTs) corresponding to the inclusion $D$ with the conductivity $\sigma_0$.
\end{definition}
The CGPTs are complex-valued linear combinations of the GPTs, where the expansion coefficients are determined by the Taylor series coefficients of $z^n$. We refer the reader to \cite{Ammari:2013:MSM, Ammari:2007:PMT} for more properties of the CGPTs.

From the expansion of the complex logarithm
$$\log(z-\tilde{z})=\log z-\sum_{n=1}^\infty \frac{1}{n}\, {\tilde{z}^n}{z^{-n}}\quad\mbox{for }|z|>|\tilde{z}|,$$
by taking the real part of the expansion, the fundamental solution to the Laplacian satisfies
\begin{align}\label{log:exp:disk}
\frac{1}{2\pi}\ln |z-\tilde{z}|
&=\frac{1}{2\pi}\ln|z|-\sum_{n=1}^\infty \frac{1}{4\pi n}\left({\tilde{z}^n}{z^{-n}}
+ \overline{{\tilde{z}^n}}\, \overline{z^{-n}}\right)\quad\mbox{for }|z|>|\tilde{z}|.
\end{align}
A real-valued entire harmonic function $H(x)$ admits the expansion 
\beq\label{eqn:Hexpan:complex}
H(x) =\alpha_0+ \sum_{m=1}^\infty \left(\alpha_m z^m+\overline{\alpha_m}\, \overline{ z^m}\right)
\eeq
with some complex coefficients $\alpha_m$. 
Using \eqnref{log:exp:disk}, we can expand \eqnref{umh1} into complex functions: 
\begin{theorem}[\cite{Ammari:2013:MSM}]\label{class:multipole}
For an entire harmonic function $H$ given by \eqnref{eqn:Hexpan:complex}, the solution $u$ to \eqnref{cond_eqn0} satisfies that, for $|z|>\sup\left\{|y|:y\in D\right\}$,
\begin{align*}
u(z)=H(z)-\sum_{n=1}^\infty \sum_{m=1}^\infty\frac{1}{4\pi n}\Bigg[ \left( \alpha_m \NN_{mn}^{(1)}+\overline{\alpha_m}\, \NN_{mn}^{(2)}
\right){z^{-n}}
+ \left(\overline{\alpha_m }\, \overline{\NN_{mn}^{(1)}} + \alpha_m \overline{\NN_{mn}^{(2)}}\right)\overline{z^{-n}}\, \Bigg].
\end{align*}
\end{theorem}

\subsection{Geometric multipole expansion and FPTs}
If $D$ is a disk centered at the origin, the associated Faber polynomials are $z^n$. Hence, \eqnref{log:exp:disk} is in fact an expansion into the Faber polynomials (and its complex conjugates) corresponding to the disk. 
For an inclusion $D$ of general shape, the complex logarithm admits the expansion \eqnref{log:Faber}. Using \eqnref{log:Faber}, we can generalize \eqnref{log:exp:disk}:
 for $z=\Psi(w)\in\CC\setminus\overline{D}$ and $\tilde{z}\in\overline{D}$, 
\begin{align}\label{log_expan_faber}
\frac{1}{2\pi}\ln|z-\tilde{z}|
&=\frac{1}{2\pi}\ln|w|-\sum_{n=1}^\infty \frac{1}{4\pi n}\left(F_n(\tilde{z})w^{-n}
+ \overline{F_n(\tilde{z})}\,\overline{ w^{-n}}\right).
\end{align}
Indeed, \eqnref{log_expan_faber} converges uniformly with respect to $|w|>\gamma$ and uniformly with respect to $\tilde{z}$ belonging to any fixed compact $F$ in the domain $\overline{D}$ \cite{Faber:1903:PE, Suetin:1998:SFP}.

Also, for an entire real harmonic function $H$, we have
\beq\label{eqn:H:faber}
H(x) =\alpha_0+ \sum_{m=1}^\infty \left(\alpha_m F_m(z)+\overline{\alpha_m}\,\overline{ F_m(z)}\right)
\eeq
for some complex coefficients $\alpha_n$. Moreover, \eqnref{eqn:H:faber} converges uniformly on any given compact domain \cite{Johnston:1987:FER}.

As one of the main contribution of this paper, we now generalize the concept of CGPTs and the classical multipole expansion \eqnref{class:multipole} by using the Faber polynomials as follows. 
 \begin{definition}\label{def:FPT}
Let $\lambda=\frac{\sigma_0+1}{2(\sigma_0-1)}$ and $F_n$ be the Faber polynomials of $D$.  For $m,n\in \NN$, we define
\begin{align*}
\FF_{mn}^{(1)}(D,\sigma_0)=\int_{\p D}F_n(z)\left(\lambda I-\Kcal^*_{\p D}\right)^{-1}\left[\pd{F_m}{\nu}\right](z)\,d\sigma(z),\\
\FF_{mn}^{(2)}(D,\sigma_0)=\int_{\p D}F_n(z)\left(\lambda I-\Kcal^*_{\p D}\right)^{-1}\left[\pd{\overline{F_m}}{\nu}\right](z)\,d\sigma(z).
\end{align*}
We call $\FF_{mn}^{(1)}$ and $\FF_{mn}^{(2)}$ the Faber polynomial polarization tensors (FPTs) corresponding to the domain $D$ with the conductivity $\sigma_0$.
\end{definition}

Let us find an expansion of the single layer potential in \eqnref{umh1}. Set $z=\Psi(w)\in\CC\setminus \overline{D}$ with $|w|>\gamma$.
It follows from \eqnref{log_expan_faber} that
\begin{align}\notag
\Scal_{\p D}[\varphi](z)
&=\frac{1}{2\pi} \int_{\p D}\ln|z-\tilde{z}| \varphi(\tilde{z})\, d\sigma(\tilde{z})\\
&=-  \int_{\p D} \sum_{n=1}^\infty  \frac{1}{4\pi n} \left( F_n(\tilde{z})w^{-n} + \overline{F_n(\tilde{z})}\,\overline{ w^{-n}}\right) \varphi(\tilde{z})\, d\sigma(\tilde{z}),\label{S:Faber:proof:thm}
\end{align}
where $\varphi$ is obtained from \eqnref{umh2} and \eqnref{eqn:H:faber} that
\beq\label{phi_derivativeH}
\varphi=(\lambda I-\Kcal_{\p D}^*)^{-1}\left[\sum_{m=1}^\infty \left(\alpha_m \frac{\partial F_m}{\partial \nu}+\overline{\alpha_m}\,\frac{\partial \overline{ F_m}}{\partial \nu}\right) \right].
\eeq
Both the infinite series in \eqnref{S:Faber:proof:thm} and \eqnref{phi_derivativeH} are uniformly convergent for $\tilde{z}\in\p D$ (with $z$ fixed). Since $(\lambda I-\Kcal_{\p D}^*)^{-1}$ is a bounded operator, we then can exchange the order of integral and summation in \eqnref{S:Faber:proof:thm} and get the desired expansion:

\begin{theorem}[Geometric multipole expansion]\label{gme}
For an entire harmonic function $H$ given by \eqnref{eqn:Hexpan:complex}, the solution $u$ to \eqnref{cond_eqn0} satisfies that, for $z=\Psi(w)\in\CC\setminus \overline{D}$ with $|w|>\gamma$,
\begin{align*}
u(z)=H(z)-\sum_{m=1}^\infty \sum_{n=1}^\infty\frac{1}{4\pi n}\Bigg[ \left( \alpha_m \FF_{mn}^{(1)}+ \overline{\alpha_m}\, \FF_{mn}^{(2)}
\right){w^{-n}}
+ \left(\overline{\alpha_m}\,\overline{ \FF_{mn}^{(1)}} + \alpha_m  \overline{\FF_{mn}^{(2)}} \right)\overline{w^{-n}}\, \Bigg].
\end{align*}
\end{theorem}
We emphasize that the geometric multipole expansion in Theorem \ref{gme} holds in the whole exterior region of $D$, differently from the classical multipole expansion in Theorem \ref{class:multipole}.

\section{Explicit matrix expression for the FPTs} \label{FPTGME}

 \subsection{Grunsky matrix $C$ and its symmetrization $G$}
 
We denote by $C$ the Grunsky matrix 
\beq\label{def:C:mat}
C=\big(c_{mn}\big)_{m,n=1}^\infty.
\eeq
We then denote by $G$ the symmetrization of the Grunsky matrix, i.e., 
\be
G=\big(g_{mn}\big)_{m,n=1}^\infty\quad\mbox{with }  g_{mn}=\sqrt{\frac{n}{m}} \frac{c_{mn}}{\gamma^{m+n}}.
\ee
 From the Grunsky identity, $g_{mn}$ satisfy the symmetry relation: $g_{mn}=g_{nm}$ for all positive integers $n$ and $m$.

Let $l^2(\CC)$ denote the vector space of the complex sequence $(v_m)$ satisfying 
$\sum_{m=1}^\infty |v_m|^2<\infty$. 
We interpret the matrix $G$ as a linear operator from $l^2(\CC)$ to $l^2(\CC)$ defined by
\begin{align*}
(v_m) &\longmapsto (w_m)\quad\mbox{with}\quad w_m = \sum_{k=1}^\infty g_{mk} v_k.
\end{align*}
Indeed, it holds from \eqnref{inequal:strong} that
\begin{equation}\label{ineq:Grunsky}
\sum_{k=1}^\infty \left| \sum_{m=1}^\infty g_{mk} v_m \right|^2 \leq \sum_{m=1}^\infty |v_m|^2.
\end{equation}
The inequality \eqnref{ineq:Grunsky} and the symmetricity of $g_{mk}$ imply
\begin{align*}
\left\lVert {G} \right\rVert^2
& = \sup_{\left\lVert (v_k) \right\rVert=1} \sum_{m=1}^\infty \left| \sum_{k=1}^\infty g_{mk} v_k \right|^2 = \sup_{\left\lVert (v_m) \right\rVert=1} \sum_{k=1}^\infty \left| \sum_{m=1}^\infty g_{mk} v_m \right|^2\leq 1.
\end{align*} 
According to \cite[Theorems 9.12-13]{Pommerenke:1975:UF}, it holds for some constant $\kappa\in[0,1)$ that
\begin{equation}\label{Gnorm}
\left\lVert {G} \right\rVert_{l^2\rightarrow l^2} \le \kappa <1
\end{equation}
since $\p \Om$ is quasiconformal; we refer the reader to \cite{Ahlfors:1963:QR, Beckermann:2018:BOP, Kuehnau:1971:VKG, Springer:1964:FEQ} for more properties of quasiconformality.

We can express $G$ in terms of $C$ as
\beq\label{G:inC}
G = \NN^{-\frac{1}{2}} \gamma^{-\NN} C \gamma^{-\NN} \NN^{\frac{1}{2}},
\eeq
where $\gamma^{\pm k\NN}$ and $\NN^{\pm\frac{1}{2}}$ denote the semi-infinite diagonal matrices whose $(n,n)$-entries are $\gamma^{\pm kn}$ and $n^{\pm\frac{1}{2}}$, respectively.

\subsection{FPTs in terms of the Grunsky matrix}

\begin{theorem}\label{F1F2ab}
 Let $D$ be a  bounded and simply connected domain in $\RR^2$ with $C^{1,\alpha}$ boundary for some $\alpha>0$, and $\lambda=\frac{\sigma_0+1}{2(\sigma_0-1)}$. The FPTs satisfy
\begin{align*}
\FF_{mn}^{(1)}(D,\sigma_0)&=4\pi n c_{mn}+ 4\pi n \left(\frac{1}{4}-\lambda^2\right) \left[ C\left( \lambda^2 I - \frac{\gamma^{-2\mathbb{N}} \overline{C}\gamma^{-2\mathbb{N}} C}{4} \right)^{-1}  \right]_{mn}, \\
\FF_{mn}^{(2)}(D,\sigma_0)&=8\pi n \lambda \gamma^{2m}\, \delta_{mn}+ 8\pi n \lambda \gamma^{2m} \left(\frac{1}{4}-\lambda^2\right) \left[ \left( \lambda^2 I - \frac{\gamma^{-2\mathbb{N}} \overline{C} \gamma^{-2\mathbb{N}} C}{4} \right)^{-1}  \right]_{mn},
\end{align*}
where $\delta_{mn}$ is the Kronecker delta function.
\end{theorem}
\begin{proof}
From \eqnref{eqn:Kstarjump} and \eqnref{Scal:expan}, we have
$$
\pd{F_m}{\nu} =-2 \sqrt{m}\gamma^m \left(-\frac{1}{2}I+\Kcal^*_{\p D}\right)\left[\zeta_m\right]\quad\mbox{on }\p D.$$
Applying \eqnref{Scal:zeta_m}, the FPTs becomes
\begin{align*}
\FF_{mn}^{(1)}
&=4\sqrt{mn}\, \gamma^{m+n}\int_{\p D}\Big(\frac{1}{2}I+\Kcal_{\p D}\Big)[\eta_n]\left[\big(\lambda I-\Kcal^*_{\p D}\big)^{-1}\Big(\frac{1}{2}I-\Kcal^*_{\p D}\Big)[\zeta_m]\right]d\sigma\\
&=4\sqrt{mn}\, \gamma^{m+n}\int_{\p D}\eta_n \left[\Big(\frac{1}{2}I+\Kcal^*_{\p D}\Big)\big(\lambda I-\Kcal^*_{\p D}\big)^{-1}\Big(\frac{1}{2}I-\Kcal^*_{\p D}\Big)[\zeta_m]\right]d\sigma\\
&=4\sqrt{mn}\, \gamma^{m+n}\int_{\p D}\eta_n \left[ \lambda I+\Kcal^*_{\p D}+\Big(\frac{1}{4}-\lambda^2\Big)\left(\lambda I-\Kcal^*_{\p D}\right)^{-1}\right][\zeta_m]\, d\sigma
\end{align*}
and
$$\FF_{mn}^{(2)}=4\sqrt{mn}\, \gamma^{m+n}\int_{\p D}\eta_n \left[ \lambda I+\Kcal^*_{\p D}+\Big(\frac{1}{4}-\lambda^2\Big)\left(\lambda I-\Kcal^*_{\p D}\right)^{-1}\right][\zeta_{-m}]\, d\sigma.$$
From the fact that  $d\sigma(z)=h(\rho_0,\theta)d\theta$, one can easily find that 
\beq\label{etazetarelation}
\int_{\p D}\eta_n\zeta_m d\sigma =0,\ 
\int_{\p D}\eta_n{\zeta_{-m}} d\sigma = 2\pi\, \delta_{mn}\quad\mbox{for all }m,n\in \NN.
\eeq
Then, by using \eqnref{NP_series1} and \eqnref{etazetarelation}, we have
\begin{align}
\FF_{mn}^{(1)}(D,\sigma_0)&
=4\pi n c_{mn}+\left(\frac{1}{4}-\lambda^2\right) A_{m,n}, \label{homogeneousF1} \\
\FF_{mn}^{(2)}(D,\sigma_0)&
=8\pi n \lambda \gamma^{2m}\, \delta_{mn}+ \left(\frac{1}{4}-\lambda^2\right) A_{-m,n} \label{homogeneousF2}
\end{align}
with 
\beq\label{Apmmn}
A_{\pm m, n}:=4\sqrt{mn}\,\gamma^{m+n} \int_{\p D}\eta_n  \left(\lambda I-\Kcal^*_{\p D}\right)^{-1}\left[\zeta_{\pm m}\right]\, d\sigma.
\eeq
In the remaining of the proof, we derive explicit expression for $A_{\pm m, n}$.

Since $\zeta_m$ form a basis of $H^{-1/2}(\p D)$, we can expand $\left(\lambda I-\Kcal^*_{\p D}\right)^{-1}\left[\zeta_{\pm m}\right]$ as 
\begin{align}
\left(\lambda I-\Kcal^*_{\p D}\right)^{-1}\left[\zeta_{m}\right] &= \sum_{n=1}^\infty \left( x_{mn} \zeta_n + y_{mn} \zeta_{-n} \right), \label{liki} \\
\left(\lambda I-\Kcal^*_{\p D}\right)^{-1}\left[\zeta_{-m}\right]&= \sum_{n=1}^\infty \left( \, \overline{x_{mn}} \zeta_{-n} + \overline{y_{mn}} \zeta_{n}\right ).\label{likiconj}
\end{align} 
for some expansion coefficients $x_{mn}$ and $y_{mn}$.
Applying $(\lambda I-\Kcal^*_{\p D})$ on both sides of \eqnref{liki}, from \eqnref{NP_series1}, we obtain
\begin{align}
\zeta_{m} 
&=  \sum_{n=1}^\infty  x_{mn} \left(\lambda \zeta_n -\frac{1}{2}\sum_{k=1}^{\infty}\, g_{nk} {\zeta}_{-k} \right) + \sum_{n=1}^\infty y_{mn} \left(\lambda  \zeta_{-n} -\frac{1}{2}\sum_{k=1}^{\infty}\, \overline{g_{nk}}\, \zeta_{k}\right), \label{zetamseries}
\end{align}
where $g_{kn}$ is the symmetrized Grunsky coefficient given by \eqnref{ineq:Grunsky}. Rearranging \eqnref{zetamseries}, we have 
\begin{align*}
&\sum_{n=1}^\infty \left(\lambda x_{mn} -\delta_{mn}\right) \zeta_n + \sum_{n=1}^\infty \lambda y_{mn} \zeta_{-n} 
=\sum_{n=1}^\infty \sum_{k=1}^{\infty} \frac{y_{mk}}{2}\, \overline{g_{kn}} \, \zeta_{n} +
 \sum_{n=1}^\infty  \sum_{k=1}^{\infty} \frac{x_{mk}}{2}\, g_{kn}\, {\zeta}_{-n}.
\end{align*}
It implies that
\begin{align}
\lambda\, x_{mn}&= \delta_{mn}+\frac{1}{2} \sum_{k=1}^{\infty} y_{mk}\, \overline{g_{kn}}, \label{x}\\
\lambda\, y_{mn} & =  \frac{1}{2} \sum_{k=1}^{\infty} x_{mk}\, g_{kn},\label{y}
\end{align}
By combining \eqnref{x} and \eqnref{y}, we have the matrix expressions for $X=(x_{mn})_{m,n=1}^\infty$ and $Y=(y_{mn})_{m,n=1}^\infty$ as follows:
\begin{align}
X & = \lambda \left( \lambda^2 I - \frac{G \overline{G}}{4} \right)^{-1}, \label{xmatrix}\\
Y & = \frac{1}{2} G \left( \lambda^2 I - \frac{\overline{G} G}{4} \right)^{-1}. \label{ymatrix}
\end{align}
The inverse matrix in \eqnref{ymatrix} is a geometric series, which converges from the following (see \eqnref{Gnorm}):
\begin{align}
\left\lVert \left( \lambda^2I - \frac{\overline{{G}}{G}}{4}  \right)^{-1} \right\rVert
 &= \lambda^{-2} \left\lVert \sum_{n=0}^\infty \left(\frac{\overline{{G}}{G}}{4\lambda^2}\right)^n\right\rVert \notag\\
 &\leq \lambda^{-2} \sum_{n=0}^\infty \left(\frac{\left\lVert{G}\right\lVert^2}{4\lambda^2} \right)^n = \lambda^{-2} \left(1-\frac{\|{G}\|^2}{4\lambda^2} \right)^{-1} < \infty. \label{invertibleGG}
\end{align}
Similarly, the inverse matrix in \eqnref{xmatrix} also converges. 

By applying \eqnref{etazetarelation}, \eqnref{Apmmn}, \eqnref{liki}, and \eqnref{ymatrix}, we obtain
\begin{align}
A_{m, n} 
&= 4\sqrt{mn}\,\gamma^{m+n} \int_{\p D}\eta_n  \left(\lambda I-\Kcal^*_{\p D}\right)^{-1}\left[\zeta_{m}\right]\, d\sigma \notag\\
&= 8\pi \sqrt{mn}\,\gamma^{m+n} y_{mn} \notag\\
&= 4\pi n \left[ C\left( \lambda^2 I - \frac{\gamma^{-2\mathbb{N}} \overline{C}\gamma^{-2\mathbb{N}} C}{4} \right)^{-1}  \right]_{mn}. \label{4pincinv}
\end{align}
We also use \eqnref{G:inC}. 
Similarly, from \eqnref{etazetarelation}, \eqnref{Apmmn}, \eqnref{likiconj}, and \eqnref{xmatrix}, it holds that
\begin{align}
A_{-m, n}
&= 4\sqrt{mn}\,\gamma^{m+n} \int_{\p D}\eta_n  \left(\lambda I-\Kcal^*_{\p D}\right)^{-1}\left[\zeta_{-m}\right]\, d\sigma \notag\\
&= 8\pi \sqrt{mn}\,\gamma^{m+n} \, \overline{x_{mn}} \notag\\
&= 8\pi n \lambda \gamma^{2m} \left[ \left( \lambda^2 I - \frac{\gamma^{-2\mathbb{N}} \overline{C} \gamma^{-2\mathbb{N}} C}{4} \right)^{-1}  \right]_{mn}. \label{8pinlg2minv}
\end{align}
From \eqnref{homogeneousF1}, \eqnref{homogeneousF2}, \eqnref{4pincinv}, and \eqnref{8pinlg2minv}, we prove the proposition.
\end{proof}
For the insulating or perfect conducting case (i.e., $\sigma_0=\infty$ or $0$), we have $\lambda=\pm 1/2$. 
It is then straightforward to derive the following lemma from Theorem \ref{F1F2ab}.
\begin{cor}\label{cor:extreme}
Let $D$ be a  bounded, simply connected, and $C^{1,\alpha}$ domain in $\RR^2$ with the conductivity $\sigma_0=\infty$ or $0$. Then, the corresponding FPTs are
\begin{align*}
\FF_{mn}^{(1)}(D,\sigma_0)&=4\pi n c_{mn},\\
\FF_{mn}^{(2)}(D,\sigma_0)&=\pm 4\pi n \gamma^{2m}\, \delta_{mn},
\end{align*}
where $+$ and $-$ corresponds to $\sigma_0=\infty$ and $\sigma_0=0$, respectively. 
\end{cor}

\subsection{Polarization tensor of an inclusion with extreme conductivity}

Plugging $m=n=1$ into Corollary \ref{cor:extreme} with $\sigma_0=\infty$ or $0$,  we arrive to the relation
\begin{align}
\FF_{11}^{(1)}(D,\sigma_0)&=4\pi c_{11}=4\pi a_1,\label{F11_1:extreme}\\
\FF_{11}^{(2)}(D,\sigma_0)&= \pm 4\pi \gamma^2.\label{F11_2:extreme}
\end{align}
Since $F_1(z)=z-a_0$, the first-order FPTs coincide with those of CGPTs. Hence, it holds that 
\begin{align*}
\FF_{11}^{(1)}&=\NN_{11}^{(1)}=m_{11}-m_{22}+i(m_{12}+m_{21}),\\
\FF_{11}^{(2)}&=\NN_{11}^{(2)}=m_{11}+m_{22},
\end{align*}
where the $2\times2$ symmetric matrix
$$
M=\left[\begin{array}{cc}
m_{11} & m_{12}\\
m_{21} & m_{22}
\end{array}\right]
$$
denotes the polarization tensor (PT) associated with $D$. In other words, 
$m_{11}=M_{(1,0)(1,0)},\ m_{12}=M_{(1,0)(0,1)},\ m_{21}=M_{(0,1)(1,0)}$, and $m_{22}=M_{(0,1)(0,1)}$,
following the definition \eqnref{gpt}. We then obtain the following lemma from \eqnref{F11_1:extreme} and \eqnref{F11_2:extreme}.

\begin{lemma}\label{lemma:PT}
Let $D$ be a  bounded and simply connected $C^{1,\alpha}$ domain in $\RR^2$ with the conductivity $\sigma_0=\infty$ or $0$. Then, the corresponding PT satisfies
$$M
=2\pi\,
\Bigg[
\begin{array}{cc}
\ds \pm \gamma^2 +\Re\{a_1\}&  \Im\{a_1\}    \\
\ds \Im\{a_1\} &  \pm \gamma^2 -\Re\{a_1\}
\end{array}
\Bigg].
$$
\end{lemma}

It is worth remaking that Lemma \ref{lemma:PT} was extended to a domain with Lipschitz boundary in \cite{Cherkaev:2021:GSE}.

\begin{cor}\label{cor:M:inv}
Under the same assumption as in Lemma \ref{lemma:PT}, we have 
\begin{align*}
\operatorname{Tr}(M^{-1})
&=
\pm \left(\pi \gamma ^2-  \frac{\pi |a_1|^2}{\gamma^2}\right)^{-1}.
\end{align*}
\end{cor}
\begin{proof}
 Let $\lambda_\alpha$ and $\lambda_\beta$ be the eigenvalues of the PT.
It then holds from Lemma \ref{lemma:PT} that
\begin{align*}
&\lambda_\alpha+\lambda_\beta=\mbox{Tr}(M)=\pm 4\pi \gamma^2,\\
&\lambda_\alpha\lambda_\beta=\mbox{det}(M)=4\pi^2\left(\gamma^4-|a_1|^2\right).
\end{align*}
From \eqnref{ineq:a1}, we have $\gamma^4-|a_1|^2\neq 0$.
Thus, we have
$$\operatorname{Tr}(M^{-1})=\frac{1}{\lambda_\alpha}+\frac{1}{\lambda_\beta}
=\pm\frac{1}{\pi}\frac{\gamma^2}{\gamma^4-|a_1|^2}.
$$
\end{proof}
The P\'{o}lya--Szeg\"{o} conjecture asserts that $\left|{\operatorname{Tr}}(M^{-1})\right|$ for an inclusion $D$ with unit area has a minimum value if and only if $D$ is a disk or an ellipse; this conjecture was proved for general conductivity case in \cite{Polya:1951:IIM}.
Corollary \ref{cor:M:inv} leads a simple alternative proof for the insulating or perfecting conducting case in two dimensions.
Indeed, the area of the domain $D$ given by the exterior conformal mapping \eqnref{eqn:extmapping} with the conformal radius $\gamma>0$ is $$0<|D|=\pi \gamma^2 -\pi\sum_{k=1}^\infty \frac{k|a_k|^2}{\gamma^{2k}}.$$
It is then straightforward to see that $|a_1|<\gamma^2$ and 
\begin{align}\label{trace_inverse}
|D|\left|{\operatorname{Tr}}(M^{-1})\right|
&=
\left(\pi \gamma^2 -\pi\sum_{k=1}^\infty \frac{k|a_k|^2}{\gamma^{2k}}\right)
\frac{1}{\pi \gamma^2-  \frac{\pi |a_1|^2}{\gamma^2}}\leq 1.
\end{align}
The equality holds in \eqnref{trace_inverse} if and only if $a_k=0$ for all $k\geq 2$, equivalently, $D$ is a disk or ellipse.

\subsection{An ellipse case}
For the case when $\Psi(w)=w+a_0+\frac{a_1}{w}$, one can easily derive from \eqnref{grunskyformula} that
\beq\label{ellipsegrunsky}
c_{mn}=\delta_{mn}\, a_1^n\quad\mbox{for all }m,n\in\NN,
\eeq
and thus,
\beq\notag
\Kcal^*_{\p D}\left[\zeta_m\right]=\frac{a_1^m}{2\gamma^{2m}}\,{\zeta_{-m}},\quad
\Kcal^*_{\p D}\left[\zeta_{-m}\right]=\frac{\overline{a_1^m}}{2\gamma^{2m}}\,\zeta_{m}.
\eeq
Hence, the space spanned by $\{\zeta_{-m},\zeta_{m}\}$ is invariant under the operator $\lambda I -\Kcal^*_{\p D}$, where
$\lambda I -\Kcal^*_{\p D}$ corresponds to the $2\times 2$ matrix
$$
\begin{bmatrix}
\lambda & \ -\frac{a_1^m}{2\gamma^{2m}}\\[2mm]
-\frac{\overline{a_1^m}}{2\gamma^{2m}}	& \lambda
\end{bmatrix}.
$$
This matrix is invertible, by \eqnref{ineq:a1}, for $|\lambda|\geq \frac{1}{2}.$
Theorem \ref{F1F2ab} leads to the following results.
\begin{lemma} For an ellipse given by $\Psi(w)=w+a_0+\frac{a_1}{w}$ with some $\gamma>0$, it holds that
$$\FF_{mn}^{(1)}(D,\sigma_0)=\FF_{mn}^{(2)}(D,\sigma_0)=0\quad\mbox{for }m\neq n,$$ 
and, for each $m\in\NN$,
\begin{align*}
\FF_{mm}^{(1)}(D,\sigma_0)&
= 4\pi m\, a_1^m\left[1+ \left(\frac{1}{4}-\lambda^2\right)
\left(\lambda^2-\frac{|a_1|^{2m}}{4\gamma^{4m}}\right)^{-1}\right], \\
\FF_{mm}^{(2)}(D,\sigma_0)&
= 8\pi  m\gamma^{2m}\lambda\left[1+ \left(\frac{1}{4}-\lambda^2\right)
\left(\lambda^2-\frac{|a_1|^{2m}}{4\gamma^{4m}}\right)^{-1}\right].
\end{align*}
\end{lemma}

%%%%
\section{Multi-coated inclusion case}\label{sec:FPT_reducing}

In this section, we extend the concept of geometric multipole expansion to multi-coated inclusions. 
We now set $\Om$ to be a multi-coated inclusion that consists of the core $D$, also denoted by $\Om_0$, and the coating layers $\Om_j$ for $j=1,\dots,L$.

\begin{figure}[H]
	\includegraphics[trim=0 50 0 30,clip,scale=0.38]{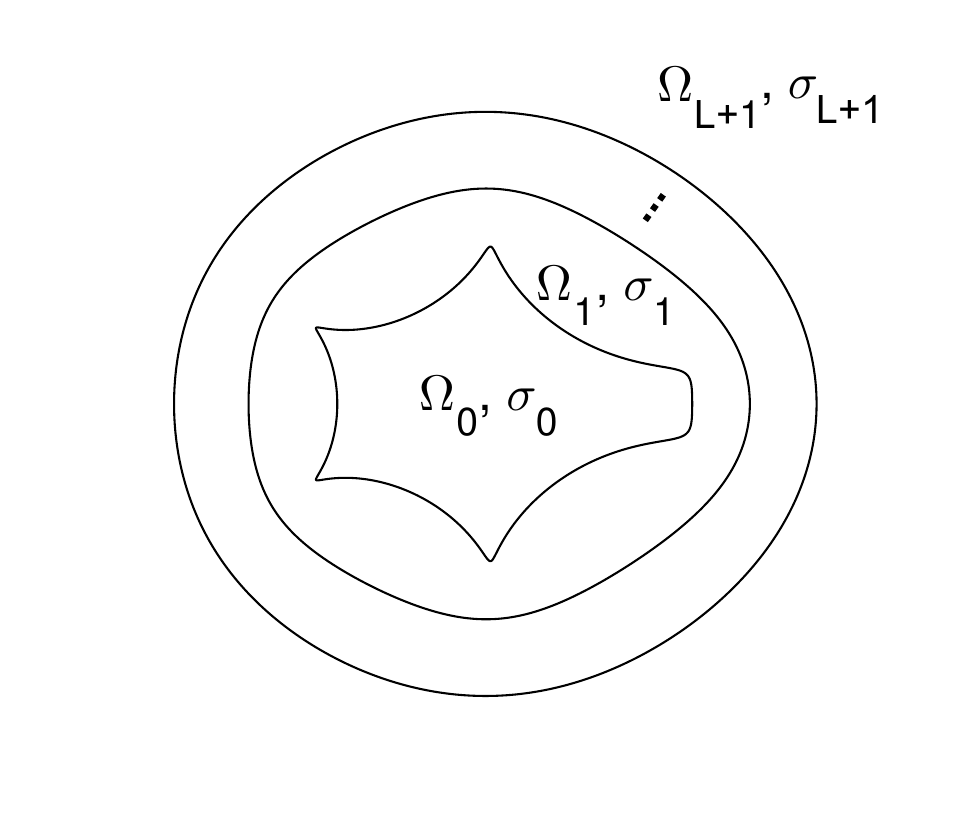}
	\caption{Multi-coated inclusion. Each coating layer is an image of concentric annuli via the exterior conformal mapping associated with the core $\Om_0$.}\label{fig:multicoated}
\end{figure}

We define $\gamma$, $\Psi(w)$ and $(\rho,\theta)$ associated with the core as in Subsection \ref{subsec:coord}. 
We assume that $\Psi$ can be conformally extended to $\{w\in\CC:|w|>\gamma-\delta\}$ for some $\delta>0$. 
We further assume that the coating layers are images of concentric annuli via the mapping $\Psi$. In other words, we have
\begin{align}
\label{def:Omj}
&\Om_j=\big\{\Psi(\rho,\theta)\,:\, r_{j-1}< e^\rho < r_j\big\},\quad j=1,2,\dots,L,\\
&\Om_{L+1}=\big\{\Psi(\rho,\theta)\,:\,  e^\rho > r_L\big\},\label{def:Om_Np1}
\end{align}
where $r_j$ are some constants satisfying $$\gamma=r_0<r_1<\cdots<r_L<\infty,$$ and $\Om_{L+1}$ is the exterior of $\Om$; we refer the reader to Figure \ref{fig:multicoated} for the geometry of the considered inclusion. 
The conductivities in $\Om_j$ are assumed to be positive constants, namely $\sigma_j$, for $j=0,1,\dots,L+1$ ($\sigma_{L+1}=1$). 
For notational simplicity, we set
\be
\boldsymbol{\sigma}=(\sigma_0,\dots,\sigma_{L}).
\ee

\subsection{Geometric multipole expansion and FPTs}
We consider the conductivity interface problem, for a given entire harmonic function $H$,
\beq\label{cond_eqn:multi}
\begin{cases}
	\ds\nabla\cdot\sigma\nabla u=0\quad&\mbox{in }\RR^2, \\
	\ds u(x) - H(x)  =O({|x|^{-1}})\quad&\mbox{as } |x| \to \infty
\end{cases}
\end{equation}
with the conductivity distribution
$$\sigma = \sigma_0\,\chi(\Om_0)+\chi(\Om_{L+1})+\sum_{j=1}^L \sigma_j \chi(\Om_j).$$
The solution $u$ satisfies the boundary condition, for each $j=0,1,\dots,L$,
\beq\label{BC:multi}
u\big|^+=u\big|^-,\quad \sigma_{j+1}\pd{u}{\nu}\Big|^+=\sigma_j\pd{u}{\nu}\Big|^-\quad\mbox{on }\p\Om_j.
\eeq

Note that $(u-H)\circ\Psi(w)$ is harmonic in $\{w:|w|>{r_L}\}$ and decays to zero as $|w|\rightarrow\infty$. 
As \eqnref{cond_eqn:multi} is linear with respect to $H(x)$, for $H$ given by $$H(x) =\alpha_0+ \sum_{m=1}^\infty \left(\alpha_m F_m(z)+\overline{\alpha_m}\,\overline{ F_m(z)}\right),$$ the solution $u$ to \eqnref{cond_eqn:multi} admits the geometric multipole expansion: for $z=\Psi(w)\in \CC\setminus\overline{\Om}$, 
\begin{align}\label{eqn:FPT_expan:multi}
u(z)=H(z)-\sum_{n=1}^\infty \sum_{m=1}^\infty\frac{1}{4\pi n}\Bigg[ \left( \alpha_m \FF_{mn}^{(1)}+ \overline{\alpha_m}\, \FF_{mn}^{(2)}\right){w^{-n}}
+ \left(\overline{\alpha_m}\,\overline{ \FF_{mn}^{(1)}} + \alpha_m\,  \overline{\FF_{mn}^{(2)}} \right)\overline{w^{-n}}\, \Bigg]
\end{align}
with some complex coefficients $\FF_{mn}^{(1)}=\FF_{mn}^{(1)}(\Om,\boldsymbol{\sigma})$ and $\FF_{mn}^{(2)}=\FF_{mn}^{(2)}(\Om,\boldsymbol{\sigma})$.
\begin{definition}
We call $\FF_{mn}^{(1)}$ and $\FF_{mn}^{(2)}$ the FPTs corresponding to the multi-coated inclusion $\Om$ with the conductivity $\boldsymbol{\sigma}$. 
We denote the FPTs in matrices:
\be
\FF^{(1)}(\Om,\boldsymbol{\sigma}) = \left(\FF^{(1)}_{mn}\right)_{m,n=1}^\infty, \quad \FF^{(2)}(\Om,\boldsymbol{\sigma}) = \left(\FF^{(2)}_{mn}\right)_{m,n=1}^\infty.
\ee
\end{definition}

\subsection{Series solutions for the transmission problem}

Fix $m\in\NN$.
To find explicit formulas of FPTs associated with a multi-coated inclusion, we express the solution, namely $u_m$, to \eqnref{cond_eqn:multi} with $H$ replaced by
$$H_m(x) = \alpha_m F_m(z) + \overline{\alpha_m}\,\overline{F_m(z)}.$$

We set $z=x_1+ix_2=\Psi(w)$ for $x\in\RR^2\setminus\overline{\Om_0}$. Since $u_m$ is harmonic in each $\Om_j$, $j=0,\dots,L+1$, and $(u_m-H_m)(x)$ decays to zero as $|x|\rightarrow\infty$, 
we can express $u_m$ as
\beq\label{multiexpan}
u_m(x) =
\begin{cases}
\ds\sum_{n=1}^\infty \left( b_{mn} F_n(z) + \overline{b_{mn}}\,\overline{F_n(z)} \right) \quad &\mbox{in } {\Om_0},\\
\ds\sum_{n=1}^\infty \left(\, \beta_{mn}^{j,1}\, w^n  +\overline{\beta_{mn}^{j,2}}\, w^{-n}+ \overline{\beta_{mn}^{j,1}}\,\overline{w^n} + {\beta_{mn}^{j,2}}\,\overline{w^{-n}}\,\right) \quad &\mbox{in }\Om_j,\quad j=1,\dots,L,\\
\ds  \alpha_m F_m(z) + \overline{\alpha_m}\,\overline{F_m(z)}+\sum_{n=1}^\infty \left(\, {s_{mn}} w^{-n} + \overline{s_{mn}}\,\overline{w^{-n}}\,\right)\quad&\mbox{in }\Om_{L+1}
\end{cases}
\eeq
with some complex coefficients $b_{mn}$, $s_{mn}$, $\beta_{mn}^{j,1}$, $\beta_{mn}^{j,2}$.
Since $\Psi$ can be conformally extended to $\{w\in\CC:|w|>\gamma-\delta\}$ for some $\delta>0$, we can expand $u_m$ into $w^{\pm n}$ for $|w|\in(\gamma-\delta,\gamma)$; see the discussion at the end of Subsection \ref{subsec:coord}. For $z\in \Om_{L+1}$ (i.e., $|w|>\rho_L$), we just apply \eqnref{eqn:Faberdefinition}. It then follows that
\begin{align*}
u_m(x)&=
\ds\sum_{n=1}^\infty \left(\, \beta_{mn}^{0,1}\, w^n  + \overline{\beta_{mn}^{0,2}}\, w^{-n}  + \overline{\beta_{mn}^{0,1}}\,\overline{w^n}+ {\beta_{mn}^{0,2}}\,\overline{w^{-n}}\,\right)\quad \mbox{for }|w|\in(\gamma-\delta,\gamma),\\
u_m(x)&=\ds\sum_{n=1}^\infty \left(\,\beta_{mn}^{L+1,1}\, w^n +\overline{\beta_{mn}^{L+1,2}}\, w^{-n}
+\overline{\beta_{mn}^{L+1,1}}\,\overline{w^n}+{ \beta_{mn}^{L+1,2}}\,\overline{w^{-n}}\,\right)\quad\mbox{in }\Om_{L+1}
\end{align*}
with
\beq \label{beta_k:coeff}
\begin{aligned}
\beta_{mn}^{0,1}=b_{mn},\quad &{\beta_{mn}^{0,2}}=\sum_{k=1}^\infty \overline{b_{mk}}\,\overline{c_{kn}},\quad \beta_{mn}^{L+1,1} = \alpha_m  \delta_{mn},\quad &\beta_{mn}^{L+1,2}=\overline{\alpha_m}\,\overline{c_{mn}}+\overline{s_{mn}}.
\end{aligned}
\eeq

By applying \eqnref{multiexpan} and \eqnref{beta_k:coeff} to the interface condition in \eqnref{BC:multi}, we obtain that, for each $j=0,1,\dots,L$, 
\beq\label{inter:eqn:eachlayer}
\begin{aligned}
\ds\beta_{mn}^{j+1,1}\, r_j^n + {\beta_{mn}^{j+1,2}}\,r_j^{-n}
&=\ds\beta_{mn}^{j,1}\, r_j^n + {\beta_{mn}^{j,2}}\,r_j^{-n},\\[1mm]
\ds\sigma_{j+1} \left(\beta_{mn}^{j+1,1}\, r_j^n -{\beta_{mn}^{j+1,2}}\, r_j^{-n}\right)
&=\ds\sigma_j\left(\beta_{mn}^{j,1}\, r_j^n -{\beta_{mn}^{j,2}}\,r_j^{-n}\right).
\end{aligned}
\eeq
We can rewrite \eqnref{inter:eqn:eachlayer} as
\beq\label{def:d_k1}
\left[
 \begin{array}{c}
\ds \beta_{mn}^{j,1} \\
\ds {\beta_{mn}^{j,2}} 
  \end{array}
  \right]
=T_{n}^j
\begin{bmatrix}
\ds \beta_{mn}^{j+1,1} \\
\ds {\beta_{mn}^{j+1,2}}
\end{bmatrix}
\quad\mbox{with }\quad
T_{n}^j=\frac{1}{2}\left(1-\frac{\sigma_{j+1}}{\sigma_j}\right)
\begin{bmatrix}
\ds 2\lambda_j & \ds r_j^{-2n}\\[2mm]
\ds r_j^{2n} & \ds 2\lambda_j
\end{bmatrix}
\eeq
with $$
\lambda_j = \frac{\sigma_{j} + \sigma_{j+1}}{2(\sigma_{j} - \sigma_{j+1})}.
$$
It then directly follows that 
\beq\label{beta:matrix:eqn}
\begin{bmatrix}
\ds \beta_{mn}^{0,1} \\[2mm]
\ds {\beta_{mn}^{0,2}} 
\end{bmatrix}
=\begin{bmatrix}
d_n^{(1)}&  d_n^{(2)}\\[2mm]
 d_n^{(3)} &  d_n^{(4)}
\end{bmatrix}
\begin{bmatrix}
\ds \beta_{mn}^{L+1,1} \\[2mm]
\ds {\beta_{mn}^{L+1,2}}
\end{bmatrix}
\quad\mbox{with }
\begin{bmatrix}
d_n^{(1)}&  d_n^{(2)}\\[2mm]
 d_n^{(3)} &  d_n^{(4)}
\end{bmatrix}
:=T_{n}^0 T_{n}^1\cdots T_{n}^L.
\eeq
Note that $d_n^{(1)},\dots,d_n^{(4)}$ are real-valued constants that are independent of $\alpha_m$ and $m$.

Since \eqnref{cond_eqn:multi} is linear with respect to the background field $H$, we can express $s_{mn}$ in \eqnref{multiexpan} as, fixing $m$ and $n$, 
$$s_{mn}=\alpha_m \, t_1 +\overline{\alpha_m}\, t_2$$
for some complex values $t_1$ and $t_2$ independent of $\alpha_m$. 
In view of the expansion \eqnref{eqn:FPT_expan:multi}, we then have
\beq\label{eqn:d_k}
s_{mn} = - \alpha_m \frac{\FF_{mn}^{(1)}}{4\pi n} - \overline{\alpha_m}\, \frac{\FF_{mn}^{(2)}}{4\pi n}.
\eeq
From \eqnref{beta_k:coeff} and \eqnref{beta:matrix:eqn}, we obtain
\begin{align}\label{b_k:eqn1}
 \alpha_m\, \delta_{mn} \, d_n^{(1)} + \left(\overline{\alpha_m}\,\overline{c_{mn}}+ \overline{s_{mn}} \right) d_n^{(2)}&= b_{mn},\\\label{b_k:eqn2}
\alpha_m\, \delta_{mn}\, d_n^{(3)} + \left(\overline{\alpha_m}\,\overline{c_{mn}}+ \overline{s_{mn}} \right) d_n^{(4)}&= \sum_{k=1}^\infty \overline{b_{mk}}\, \overline{c_{kn}}.
\end{align}
As $d_n^{(1)},\dots,d_n^{(4)}$ are real-valued, it follows by plugging \eqnref{b_k:eqn1} into \eqnref{b_k:eqn2} that 
\begin{align}
\alpha_m\, \delta_{mn}\, d_n^{(3)} + (\overline{\alpha_m}\,\overline{c_{mn}}+ \overline{s_{mn}}) d_n^{(4)}
&= \sum_{k=1}^\infty \overline{b_{mk}} \overline{c_{kn}} \notag\\ 
&= \sum_{k=1}^\infty \left(\overline{\alpha_m}\, \delta_{mk}\,d_k^{(1)} \, \overline{c_{kn}} + \left(\alpha_m \,c_{mk} +  s_{mk}\right) \, d_k^{(2)} \overline{c_{kn}} \right).\label{coeff:sum:eqn}
\end{align}

\subsection{Asymptotic behavior of $d_n^{(2)}$ and $d_n^{(4)}$}

For notational convenience, we denote the multiplication of the scaling constant in \eqnref{def:d_k1} by
\be
\tau_L := \frac{1}{2^{L+1}}\left(1-\frac{\sigma_1}{\sigma_0}\right) \left(1-\frac{\sigma_2}{\sigma_1}\right) \cdots \left(1-\frac{\sigma_{L+1}}{\sigma_L}\right)
 = \frac{1}{\left( 2\lambda_0+1 \right) \left( 2\lambda_1+1 \right) \cdots \left( 2\lambda_L+1 \right)}.
\ee
It follows from determinants of the $2\times2$-matrix in \eqnref{def:d_k1} and \eqnref{beta:matrix:eqn} that
\begin{align}
d_n^{(1)} d_n^{(4)} - d_n^{(2)} d_n^{(3)} = \tau_L^2 \left( 4\lambda_0^2-1 \right) \left( 4\lambda_1^2-1 \right) \cdots \left( 4\lambda_L^2-1 \right) = \frac{\sigma_1 \sigma_2 \cdots \sigma_{L+1}}{\sigma_0 \sigma_1 \cdots \sigma_L} = \frac{\sigma_{L+1}}{\sigma_0}.\label{detD}
\end{align}
For the case $L=1$, we obtain 
\be
\begin{aligned}
d_n^{(1)} &= \tau_1 \left(4\lambda_0\lambda_1 + r_0^{-2n} r_1^{2n}\right) = \tau_1 r_0^{-2n} r_1^{2n} + o\left(r_0^{-2n} r_1^{2n}\right),\\[2mm]
d_n^{(2)} &= \tau_1 \left(2\lambda_0 r_1^{-2n} + 2 \lambda_1 r_0^{-2n} \right) = 2 \tau_1 \lambda_1 r_0^{-2n} + o\left(r_0^{-2n}\right),\\[2mm]
d_n^{(3)} &= \tau_1 \left(2\lambda_0 r_1^{2n} + 2 \lambda_1 r_0^{2n} \right) = 2 \tau_1 \lambda_0 r_1^{2n} + o\left(r_1^{2n}\right),\\[2mm]
d_n^{(4)} &= \tau_1 \left(4\lambda_0\lambda_1 + r_0^{2n} r_1^{-2n} \right) = 4 \tau_1 \lambda_0\lambda_1 + o(1),
\end{aligned}
\ee
where $o(\cdot)$ denotes the standard little-$\rm o$ notation.
If $L=2$, we have 
\begin{align*}
d_n^{(1)} &= \tau_2 \left( 8\lambda_0 \lambda_1 \lambda_2 + 2\lambda_0 r_1^{-2n}r_2^{2n} + 2 \lambda_1 r_0^{-2n} r_2^{2n} + 2 \lambda_2 r_0^{-2n} r_1^{2n} \right) = 2 \tau_2 \lambda_1 r_0^{-2n} r_2^{2n} + o\left(r_0^{-2n} r_2^{2n}\right),\\[2mm]
d_n^{(2)} &= \tau_2 \left( 4\lambda_0\lambda_1 r_2^{-2n} + 4\lambda_0 \lambda_2 r_1^{-2n} + 4 \lambda_1 \lambda_2 r_0^{-2n} + r_0^{-2n} r_1^{2n} r_2^{-2n}\right) = 4 \tau_2 \lambda_1 \lambda_2 r_0^{-2n} + o\left(r_0^{-2n}\right),\\[2mm]
d_n^{(3)} &= \tau_2 \left( 4\lambda_0 \lambda_1 r_2^{2n} + 4 \lambda_0 \lambda_2 r_1^{2n} + 4 \lambda_1 \lambda_2 r_0^{2n} + r_0^{2n} r_1^{-2n}r_2^{2n}\right) = 4 \tau_2 \lambda_0 \lambda_1 r_2^{2n}  + o\left(r_2^{2n}\right),\\[2mm]
d_n^{(4)} &= \tau_2 \left( 8\lambda_0 \lambda_1 \lambda_2 + 2\lambda_0 r_1^{2n} r_2^{-2n} + 2 \lambda_1 r_0^{2n} r_2^{-2n} + 2  \lambda_2 r_0^{2n} r_1^{-2n} \right) = 8 \tau_2 \lambda_0 \lambda_1 \lambda_2 + o(1).
\end{align*}
For general $L$, one can show the following.
\begin{lemma}\label{asymptoticD}
For each fixed $L$, $d_n^{(1)},\dots,d_n^{(4)}$ satisfy the asymptotic behavior as $n\rightarrow\infty$:
\begin{align*}
d_n^{(1)} &= 2^{L-1 } \tau_L \lambda_1 \lambda_2 \cdots \lambda_{L-1} r_0^{-2n} r_L^{2n} + o(r_0^{-2n} r_L^{2n}), \\[2mm]
d_n^{(2)} &= 2^L \tau_L \lambda_1 \lambda_2 \cdots \lambda_L r_0^{-2n} + o(r_0^{-2n}),\\[2mm]
d_n^{(3)} &= 2^L \tau_L \lambda_0 \lambda_1 \cdots \lambda_{L-1} r_L^{2n} + o(r_L^{2n}), \\[2mm]
d_n^{(4)} &= 2^{L+1} \tau_L \lambda_0 \lambda_1 \cdots \lambda_L + o(1).
\end{align*}
\end{lemma}

We denote by $D_j$ the diagonal matrix with the diagonal entries $d_n^{(j)}$, that is,
\beq \label{def:C:D2:only} 
D_j=\left(d_n^{(j)} \delta_{mn}\right)_{m,n=1}^\infty,\quad 1\leq j\leq 4.
\eeq
By Lemma \ref{asymptoticD}, we have $d_n^{(4)}\neq 0$ for sufficiently large $n$. Furthermore, it holds for $n\gg1$  that
\beq\label{eqn:inv:tildeG}
I-D_2 \overline{C} D_4^{-1} D_2 {C}  D_4^{-1} \approx \lambda_0^{-2} \left(\lambda_0^2 I - \frac{r_0^{-2\NN}\overline{ C} r_0^{-2\NN} {C}}{4} \right),
\eeq
where $C$ denotes the Grunsky matrix (see \eqnref{def:C:mat}). 
Note that the right-hand side of \eqnref{eqn:inv:tildeG} is invertible from \eqnref{invertibleGG}. 
 For all examples in Subsection \ref{sec:numer:examples}, the $100\times 100$ dimensional truncated matrices of $D_4$ and $D_4 - D_2 \overline{C} D_4^{-1} D_2 {C}$ are invertible.

From now on, we assume that $d_n^{(4)}\neq 0$ for all $n$ so that $D_4$ is invertible. We also assume that $D_4 - D_2 \overline{C} D_4^{-1} D_2 {C}$ is invertible.

\subsection{Matrix formulation for FPTs}

We set the two semi-infinite matrices:
\be
\boldsymbol{\alpha}=\big(\alpha_n \delta_{mn}\big)_{m,n=1}^\infty,\quad S=\big(s_{mn}\big)_{m,n=1}^\infty.
\ee
It then follows from \eqnref{coeff:sum:eqn} that
\begin{gather}
\label{fmn:1}
\ds \boldsymbol{\alpha} D_3 + (\overline{\boldsymbol{\alpha}}\, \overline{C} + \overline{S}) D_4=\ds\overline{\boldsymbol{\alpha}}\, D_1 \overline{C} + \left(\boldsymbol{\alpha} C+ S \right) D_2 \overline{C},\\
\label{fmn:2}
\ds \boldsymbol{\alpha} C + S = \ds \boldsymbol{\alpha} D_1 C D_4^{-1} + \left(\overline{\boldsymbol{\alpha}} \overline{C}+ \overline{S} \right) D_2 C D_4^{-1} -  \overline{\boldsymbol{\alpha}} D_3 D_4^{-1}.
\end{gather}
Here, \eqnref{fmn:2} is a direct consequence of \eqnref{fmn:1} by taking complex conjugates.
Substituting \eqnref{fmn:2} into \eqnref{fmn:1}, we obtain
\begin{gather}\notag
\left(\overline{\boldsymbol{\alpha}}\, \overline{C} + \overline{S}\right)\left(D_4-D_2 C D_4^{-1} D_2 \overline{C}\right)
= \overline{\boldsymbol{\alpha}}\, D_1 \overline{C} + \boldsymbol{\alpha} D_1 C D_4^{-1} D_2 \overline{C}  -  \overline{\boldsymbol{\alpha}}\, D_3 D_4^{-1} D_2 \overline{C} - \boldsymbol{\alpha} D_3,
\end{gather}
and thus,
\begin{align}\notag
\overline{S} = \left( \overline{\boldsymbol{\alpha}}\, D_1 \overline{C} + \boldsymbol{\alpha} D_1 C D_4^{-1} D_2 \overline{C}  -  \overline{\boldsymbol{\alpha}}\, D_3 D_4^{-1} D_2 \overline{C} - \boldsymbol{\alpha} D_3 \right) \left(D_4-D_2 C D_4^{-1} D_2 \overline{C}\right)^{-1} - \overline{\boldsymbol{\alpha}} \,\overline{C}.
\end{align}

Now, applying \eqnref{eqn:d_k}, we derive
\begin{align*}
&-\left(\overline{\boldsymbol{\alpha}}\, \overline{\FF^{(1)}} + \boldsymbol{\alpha}\, \overline{\FF^{(2)}}\right) (4\pi \NN)^{-1} \\
&= \left( \,\overline{\boldsymbol{\alpha}}\, D_1 \overline{C} + \boldsymbol{\alpha} D_1 C D_4^{-1} D_2 \overline{C}  -  \overline{\boldsymbol{\alpha}} \, D_3 D_4^{-1} D_2 \overline{C} - \boldsymbol{\alpha} D_3 \right) \left(D_4-D_2 C D_4^{-1} D_2 \,\overline{C}\,\right)^{-1} - \overline{\boldsymbol{\alpha}}\, \overline{C},
\end{align*}
or equivalently,
\begin{align}\label{alphaF12}
&\boldsymbol{\alpha} \, \left( \left(D_3-D_1 C D_4^{-1} D_2 \,\overline{C}\,\right) \left(D_4-D_2 C D_4^{-1} D_2 \,\overline{C}\,\right)^{-1}  - \overline{\FF^{(2)}} \left(4\pi \NN\right)^{-1}\, \right) \notag\\ 
=&\, \overline{\boldsymbol{\alpha}} \, \left( \left(D_1  D_4 - D_2 D_3\right) \, D_4^{-1}\, \overline{C} \, \left(D_4-D_2 C D_4^{-1} D_2 \overline{C}\right)^{-1} - \overline{C} +  \overline{\FF^{(1)}} \left(4\pi \NN\right)^{-1} \right).
\end{align}
Since $\boldsymbol{\alpha}$ is diagonal and corresponds to an arbitrary entire function $H$, both sides of \eqnref{alphaF12} vanish. Furthermore, from \eqnref{detD}, the diagonal entry of $ D_1  D_4 - D_2 D_3$ is as follows: for each $m\in\NN$, 
$$
\big[ D_1  D_4 - D_2 D_3 \big]_{mm} = d_m^{(1)}d_m^{(4)} - d_m^{(2)} d_m^{(3)} = \frac{\sigma_{L+1}}{\sigma_0}.
$$
Hence, we have the following theorem.
\begin{theorem}\label{FPTsystem}
We set $C$ and $D_j$ as in \eqnref{def:C:mat} and \eqnref{def:C:D2:only}, respectively. 
The FPTs of the multi-coated structure $\Om$ given as at the beginning of Section \ref{sec:FPT_reducing} have the following matrix expressions, given that $D_4$ and $D_4 - D_2 \overline{C} D_4^{-1} D_2 C$ are invertible: for each $m,n\in\NN$,
\begin{align*}
\FF_{mn}^{(1)}(\Omega,\boldsymbol{\sigma}) &= 4\pi n \, c_{mn} - 4\pi n \, \sigma_0^{-1} \sigma_{L+1} \left[ D_4^{-1} C \, \left(D_4-D_2\, \overline{C} \, D_4^{-1} D_2 \, C \right)^{-1} \right]_{mn},\\
\FF_{mn}^{(2)}(\Omega,\boldsymbol{\sigma}) &= 4\pi n \left[\left(D_3-D_1\, \overline{C}\, D_4^{-1} D_2\, C \right) \left(D_4-D_2 \,\overline{C} \, D_4^{-1} D_2\, C \right)^{-1}\right]_{mn}.
\end{align*}
\end{theorem}
When $\Om_0$ is a disk, the fact that $a_1=0$ and \eqnref{ellipsegrunsky} imply $c_{mn}=0$ for all $m,n\in\NN$. Thus, $C$ is identical to the zero matrix. Hence, we have the following corollary.
\begin{cor}[\cite{Ammari:2013:ENC1}]
If $\Om_0$ is a disk, then the FPTs of the multi-coated inclusion $\Omega$ satisfy 
$$\FF_{mn}^{(1)}(\Omega,\boldsymbol{\sigma})=0,\quad 
\FF_{mn}^{(2)}(\Omega,\boldsymbol{\sigma}) = 4\pi m \, \frac{d_m^{(3)}}{d_m^{(4)}}\delta_{mn}\quad\mbox{for all }m,n\in\NN.
$$
\end{cor}

The Grunsky coefficient of an ellipse satisfy \eqnref{ellipsegrunsky}, which means $C$ and FPTs have diagonal matrix forms.
\begin{cor}\label{cor:ellipse:FPT}
If $\Om_0$ is an ellipse, the FPTs of the multi-coated ellipse $\Omega$ are diagonal matrices:
\begin{align*}
\FF_{mn}^{(1)}(\Omega,\boldsymbol{\sigma}) &=  4\pi m \, a_1^m \delta_{mn}-  \frac{4\pi m \, a_1^m \sigma_0^{-1} \sigma_{L+1} }{ \big(d_m^{(4)}\big)^2 - \big(d_m^{(2)}\big)^2 |a_1|^{2m}}\delta_{mn},\\
\FF_{mn}^{(2)}(\Omega,\boldsymbol{\sigma}) &= 4\pi m \,\frac{d_m^{(3)} d_m^{(4)}-d_m^{(1)} d_m^{(2)} |a_1|^{2m}}{   \big(d_m^{(4)}\big)^2 - \big(d_m^{(2)}\big)^2 |a_1|^{2m}}\delta_{mn}\qquad\mbox{for all }m,n\in\NN.
\end{align*}
In particular, for $L=1$, it holds that, for each $m\in\NN$,
\begin{align*}
\FF_{mm}^{(1)}(\Omega,\boldsymbol{\sigma}) &=\ds 4\pi m \, a_1^m - \frac{4\pi m \, a_1^m (4\lambda_0^2-1)(4\lambda_1^2-1)}{\left(4\lambda_0\lambda_1 + \frac{r_0^{2m}}{r_1^{2m}} \right)^2 - 4 \left(\frac{\lambda_0}{r_1^{2m}} + \frac{\lambda_1}{r_0^{2m}}\right)^2 |a_1|^{2m} },\\[2mm]
\FF_{mm}^{(2)}(\Omega,\boldsymbol{\sigma}) &= \ds 8\pi m \frac{ \left( 4\lambda_0\lambda_1 + \frac{r_0^{2m}}{r_1^{2m}} \right) \left( \lambda_0 r_1^{2m} + \lambda_1  r_0^{2m}\right) - \left(4\lambda_0\lambda_1 + \frac{r_1^{2m}}{r_0^{2m}}  \right) \left( \frac{\lambda_0}{r_1^{2m}} + \frac{\lambda_1}{r_0^{2m}} \right) |a_1|^{2m} }{\left( 4\lambda_0\lambda_1 + \frac{r_0^{2m}}{r_1^{2m}} \right)^2 - 4 \left( \frac{\lambda_0}{r_1^{2m}} + \frac{\lambda_1}{r_0^{2m}} \right)^2 |a_1|^{2m}}.
\end{align*}
\end{cor}

\begin{remark}
For the simply connected inclusion (that is, $L=0$), we get
\[
\begin{bmatrix}
\ds d_n^{(1)} & \ds d_n^{(2)}\\
\ds d_n^{(3)} & \ds d_n^{(4)}
\end{bmatrix}
=
\frac{1}{2\lambda+1}
\begin{bmatrix}
\ds 2\lambda & \ds \gamma^{-2n}\\
\ds \gamma^{2n} & \ds 2\lambda
\end{bmatrix}
\quad 
\mbox{with } \lambda = \frac{\sigma_0+1}{2(\sigma_0-1)},
\]
and hence,
$$D_1 = D_4 = \left(\frac{2\lambda}{2\lambda+1}\right) I, \qquad D_2 = \left(\frac{1}{2\lambda+1}\right) \gamma^{-2\mathbb{N}},  \qquad  D_3 = \left(\frac{1}{2\lambda+1}\right) \gamma^{2\mathbb{N}}.$$
Applying above relations to Theorem \ref{FPTsystem}, we obtain the same results as Theorem \ref{F1F2ab}.
\end{remark}

\subsection{Domains with rotational symmetry}
Let us consider some properties of FPTs for a multi-coated structure with rotational symmetry. 
Let $N$ be a positive integer. A domain is said to have rotational symmetry of order $N$ if it is invariant under the rotation by  by $2\pi/N$ (see Figure \ref{rot_sym}).
\begin{figure}[b!]
\centering
\includegraphics[scale=0.33,trim={0.9cm 0.9cm 0.9cm 0.9cm}, clip]{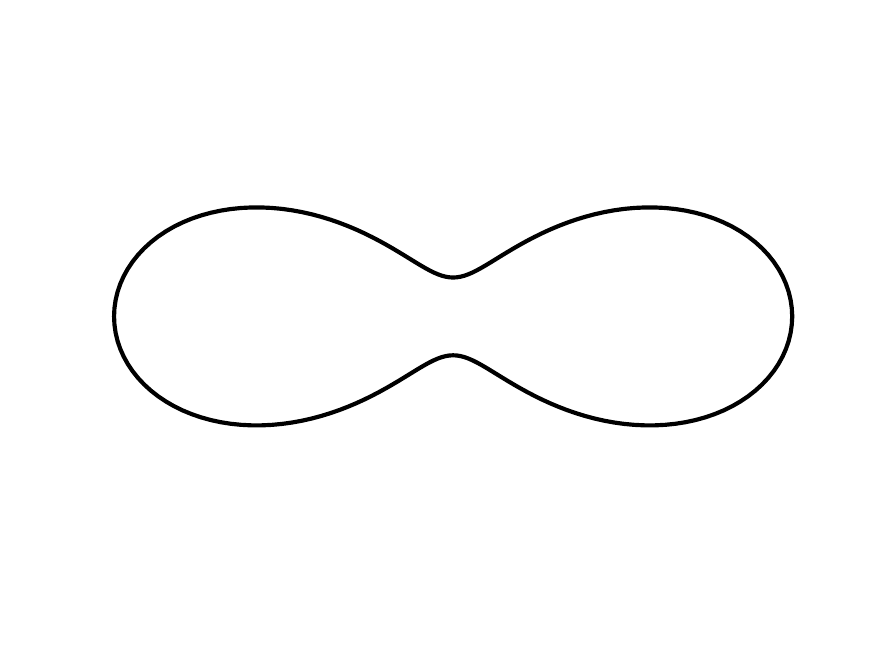}
\includegraphics[scale=0.33,trim={0.9cm 0.9cm 0.9cm 0.9cm}, clip]{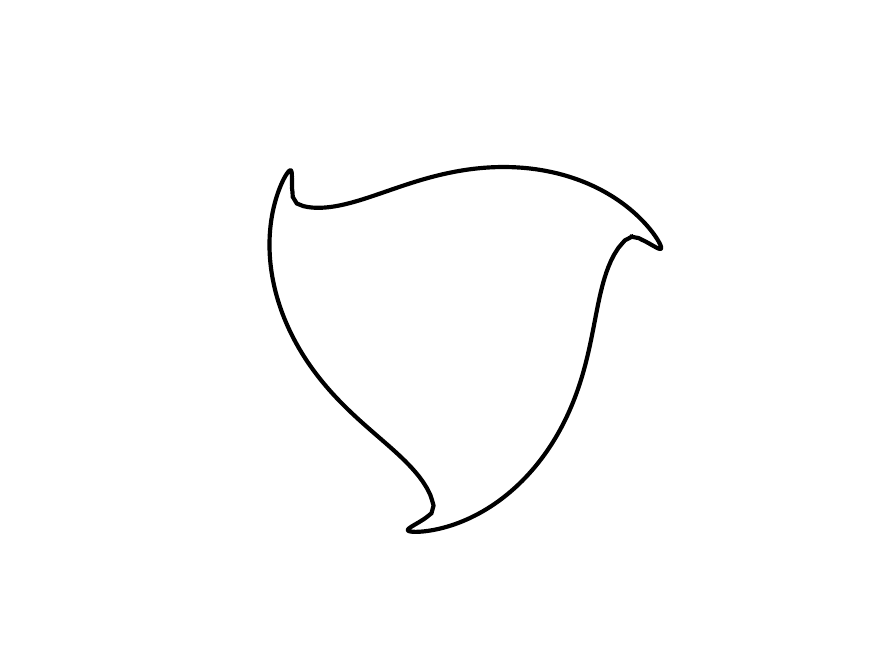}
\includegraphics[scale=0.33,trim={0.9cm 0.9cm 0.9cm 0.9cm}, clip]{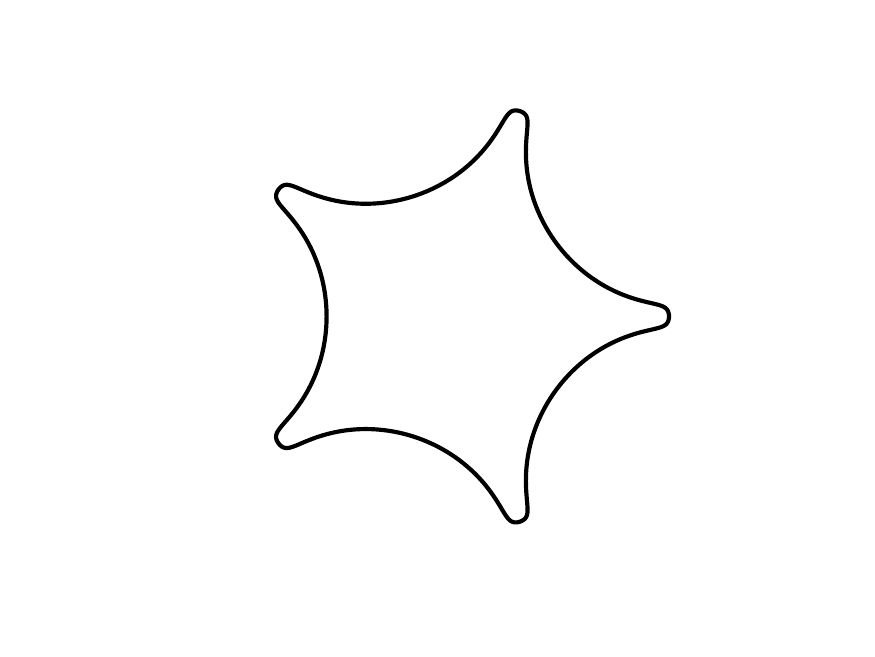}
\caption{Domains with rotational symmetry of order $2$, $3$, and $5$}\label{rot_sym}
\end{figure}
Then, the exterior conformal mapping $\Psi$ corresponding to the domain satisfies
\be
\Psi(w) = e^{\frac{2\pi i}{N}}\Psi\left(e^{-\frac{2\pi i}{N}}w\right) \quad \mbox{on } \lvert w \rvert = \gamma,
\ee
which is equivalent to
\beq\notag
a_n = 0 \quad \mbox{for } n+1\not\equiv 0 \ (\mbox{mod } N).
\eeq
Here, $m\not\equiv 0$ (mod $N$) means that $m=kN+r$ for some $k\in\NN$ and $r=1,\dots,N-1$.
It follows by induction using \eqnref{grunskyformula} that the Grunsky coefficients satisfy
\beq\label{cmn:SNm}
c_{mn} = 0 \quad \mbox{for } m+n \not\equiv 0 \ (\mbox{mod } N),
\eeq

Let us define two collections of semi-infinite matrices. The first is a set of diagonally striped infinite matrices 
\begin{align*}
\mathcal{S}^+_N := \left\{ S = (S_{mn})_{m,n\in\NN} : S_{mn} = 0  \quad \mbox{for } m-n \not\equiv 0 \  (\mbox{mod } N) \right\},
\end{align*}
and 
the second is a set of anti-diagonally striped infinite matrices
\begin{align*}
\mathcal{S}^-_N :=\left \{ S= (S_{mn})_{m,n\in\NN} : S_{mn} = 0  \quad \mbox{for } m+n \not\equiv 0 \ (\mbox{mod } N)\right \}.
\end{align*}
For instance, the following matrices $A$ and $B$ belong to $\mathcal{S}^+_3$ and $\mathcal{S}^-_3$, respectively:
\be
A=
\left[
\begin{array}{*7{C{1.5em}}}
* & \ 0 &\ 0 &\ * &\ 0 &\ 0 & \cdots\\
0 & \ * &\ 0 &\ 0 &\ * &\ 0 & \cdots\\
0 &\  0 &\ * &\ 0 &\ 0 &\ * & \cdots\\
* &\  0 &\ 0 &\ * &\ 0 &\ 0 & \cdots\\
0 &\  * &\ 0 &\ 0 &\ * &\ 0 & \cdots\\
0 &\  0 &\ * &\ 0 &\ 0 &\ * &  \cdots\\
\vdots &\ \vdots &\ \vdots &\ \vdots &\ \vdots &\ \vdots &\ \ddots 
\end{array}
\right],
\quad
B=
\left[
\begin{array}{*7{C{1.5em}}}
0 &\ * &\ 0 &\ 0 &\ * &\ 0 & \cdots\\
* &\ 0 &\ 0 &\ * &\ 0 &\ 0 & \cdots\\
0 &\ 0 &\ * &\ 0 &\ 0 &\ * & \cdots\\
0 &\ * &\ 0 &\ 0 &\ * &\ 0 & \cdots\\
* &\ 0 &\ 0 &\ * &\ 0 &\ 0 & \cdots\\
0 &\ 0 &\ * &\ 0 &\ 0 &\ * & \cdots\\
\vdots &\ \vdots &\ \vdots &\ \vdots &\ \vdots &\ \vdots &\ \ddots 
\end{array}
\right],
\ee
where $*$ represents elements that can be nonzero.
One can easily find that the following product rules hold:
\begin{lemma}\label{prodproperty}
For $U^+,V^+\in\mathcal{S}^+_N$ and $U^-,V^-\in \mathcal{S}^-_N$, we have
\be
\begin{aligned}
&\ds U^+V^+,\, U^-V^- \in \mathcal{S}^+_N,\\
&\ds U^+U^-,\, U^-U^+ \in \mathcal{S}^-_N.
\end{aligned}
\ee
\end{lemma}

\begin{prop}\label{rotationsym}
Let $\Om$ be a multi-coated inclusion given as at the beginning of this section. 
If the core $\Om_0$ has a rotational symmetry of order $N$, it follows that
\begin{align*}
&\FF_{mn}^{(1)}(\Omega, \boldsymbol{\sigma}) = 0 \quad \mbox{for } m+n \not\equiv 0 \ (\mbox{mod } N),\\
&\FF_{mn}^{(2)}(\Omega, \boldsymbol{\sigma}) = 0 \quad \mbox{for } m-n \not\equiv 0 \ (\mbox{mod } N).
\end{align*}
\end{prop}
\pf
From \eqnref{cmn:SNm}, we have 
\beq\label{csm}
C \in \mathcal{S}^-_N.
\eeq
Since $D_j$ in Theorem \ref{FPTsystem} is a diagonal matrix, we have $D_j, D^{-1}_j \in \mathcal{S}^+_N$. It follows from Lemma \ref{prodproperty} and  \eqnref{csm} that 
\beq\label{splus}
\begin{aligned}
D_3-D_1\, \overline{C}\, D_4^{-1} D_2\, C &\in \mathcal{S}^+_N, \\
D_4-D_2 \,\overline{C} \, D_4^{-1} D_2\, C &\in \mathcal{S}^+_N,
\end{aligned}
\eeq
and
$
D_2 \, C \, D_4^{-1} \in \mathcal{S}^-_N.
$
Hence, we have $$(D_2 \,\overline{C} \, D_4^{-1} D_2\, C \, D_4^{-1})^k  \in \mathcal{S}^+_N.$$
Thus, the geometric series on $\left( I - D_2 \,\overline{C} \, D_4^{-1} D_2\, C \, D_4^{-1} \right)^{-1}$ implies that
\begin{align}
\left( D_4-D_2 \,\overline{C} \, D_4^{-1} D_2\, C \right)^{-1}
&= D_4^{-1} \left( I - D_2 \,\overline{C} \, D_4^{-1} D_2\, C \, D_4^{-1} \right)^{-1} \notag\\
&= D_4^{-1} \sum_{k=0}^{\infty} \left( D_2 \,\overline{C} \, D_4^{-1} D_2\, C \, D_4^{-1} \right)^{k}\in \mathcal{S}^+_N. \label{sminus}
\end{align}
By using \eqnref{splus}, \eqnref{sminus}, and Theorem \ref{FPTsystem}, we finally obtain that
\begin{align*}
\left[\frac{\FF_{mn}^{(1)}(\Omega,\boldsymbol{\sigma})}{4\pi n} \right]_{m,n=1}^\infty &=  C - \sigma_0^{-1} \sigma_{L+1} D_4^{-1} C \, \left(D_4-D_2\, \overline{C} \, D_4^{-1} D_2 \, C \right)^{-1} \in \mathcal{S}^-_N,\\
\left[\frac{\FF_{mn}^{(2)}(\Omega,\boldsymbol{\sigma})}{4\pi n} \right]_{m,n=1}^\infty  &= \left(D_3-D_1\, \overline{C}\, D_4^{-1} D_2\, C \right) \left(D_4-D_2 \,\overline{C} \, D_4^{-1} D_2\, C \right)^{-1} \in \mathcal{S}^+_N,
\end{align*}
which complete the proof. 
\qed

%%%%%%%%%%%%%%%%%%%%%%%%%%%%%%%%%%%%%%%%%%%%%%%%%%%%%%%%%%%%%%
%%%%%%%%%%%%%%%%%%%%%%%%%%%%%%%%%%%%%%%%%%%%%%%%%%%%%%%%%%%%%%

\section{Construction of semi-neutral inclusions}\label{sec:numerical}
In this section, as an application of FPTs, we construct semi-neutral inclusions that are layered structures whose layers, except for the core, are images of concentric annuli via the exterior conformal mapping of the core (see Definition \ref{def:semi-neutral} below). For such a multi-coated inclusion, we consider mainly the material parameters in the layers, differently from the approach in \cite{Feng:2017:CGV}, in which the shapes of coating layers were determined by a shape optimization procedure.

For the concentric disks, $\FF^{(1)}=0$ \cite{Ammari:2013:ENC1}. In view of the expressions in Theorem \ref{FPTsystem} with the Grunsky matrix $C$ associated with the core $\Om_0$, one can deduce that $\FF^{(1)}$ significantly depends on the shape of $\Om_0$. Hence, one cannot find a neutral inclusion of the considered multilayered geometry, except the concentric disks. Instead, we can obtain semi-neutral inclusions with the core of general shape that are not perfectly neutral but show relatively negligible field perturbations for low-order polynomial loadings. 

For the simply connected domain of general shape, $\FF^{(1)}$ is not zero but still significantly smaller than $\FF^{(2)}$, and hence, $\FF^{(2)}$ mainly contributes in the field perturbation. 
For a given core $\Om_0$ of general shape, one can construct layered structures that show relatively small field perturbation. We choose the number of layers and determine appropriate conductivity values in the coating layers such that the multi-coated inclusions satisfy the following condition.
\begin{definition}[Semi-neutral inclusion] \label{def:semi-neutral}
Let $\Om$ be a multi-coated inclusion with the core $\Om_0$ and $L$ layers of coating, as in Section \ref{sec:FPT_reducing}. We say that $\Om$ is a semi-neutral inclusion provided that, for some $N_1, N_2\in\NN$, the second FPTs are negligible for low-order terms, that is, 
\be
\frac{\FF^{(2)}_{mn}(\Omega,\boldsymbol{\sigma})}{4\pi n} \approx 0\quad\mbox{for }\, m\le N_1 \mbox{ and } n\leq N_2.
\ee
\end{definition}

\subsection{Numerical scheme}\label{subsec:FPT_reducing}

For a given $\Om_0$, we denote $\gamma$, $\Psi(w)$, and $(\rho,\theta)$ as in Subsection \ref{subsec:coord}. 
We set $\Om_j$ ($j=1,\dots,L+1$) as in \eqnref{def:Omj} and \eqnref{def:Om_Np1}. 
We assume that $\sigma_{L+1}=1$ and that the conductivity $\sigma_0$ is fixed. 

We find $\boldsymbol{\sigma} = (\sigma_1,\dots,\sigma_L)$ such that the condition Definition \ref{def:semi-neutral} holds.
In other words, we look for $\boldsymbol{\sigma}$ that satisfies the equation
\beq\label{fsig:condition}
\boldsymbol{f}(\boldsymbol{\sigma})\approx 0,
\eeq
where $\boldsymbol{f}:(\RR^+)^{L}\to\RR^{N_1N_2}$ is a nonlinear vector-valued function $\boldsymbol{\sigma}\longmapsto (f_1,...,f_{N_1N_2})$
given by 
$$
f_l =  \left|\frac{\FF^{(2)}_{mn}}{4\pi n}\right|,\quad l=N_2(m-1)+n,
$$
for $1 \le m \le N_1$, $1\le n \le N_2$. To find $\boldsymbol{\sigma}$ satisfying \eqnref{fsig:condition}, we use the multivariate Newton's method:
\beq\label{initialguess}
\boldsymbol{\sigma}^{(k+1)} = \boldsymbol{\sigma}^{(k)} - \alpha\, \boldsymbol{J}^{\dagger}\big[\boldsymbol{\sigma}^{(k)}\big]\, \boldsymbol{f}\big[\boldsymbol{\sigma}^{(k)}\big],
\eeq
where $\alpha$ is a constant in $(0,1)$, and $k$ indicates the iteration step, and  $\boldsymbol{J}^\dagger$ denotes the pseudo--inverse of the Jacobian matrix of $\boldsymbol{f}$. We give the initial guess $\boldsymbol{\sigma}^{(0)}$ as a proper alternating series.

We calculate $\boldsymbol{f}(\boldsymbol{\sigma}^{(k)})$ by Theorem \ref{FPTsystem}, where we truncate the semi-infinite matrices $C, D_2, D_4$ to be $100\times100$ matrices. To obtain $\boldsymbol{J}^{\dagger}[\boldsymbol{\sigma}^{(k)}]$, we use the finite difference approximation of the partial derivatives of $\boldsymbol{f}$. All the computations in this paper are performed by Matlab R2020a software.
 We iterate the multivariate Newton's method until the conductivity distribution satisfies the stopping condition:
$$
\frac{|\boldsymbol{\sigma}^{(k+1)}-\boldsymbol{\sigma}^{(k)}|}{|\boldsymbol{\sigma}^{(k)}|} < 10^{-15}.
$$

\subsection{Numerical examples}\label{sec:numer:examples}

\subsubsection{Elliptical coated inclusions}\label{sec:coated ellipse}

Let $\Om$ be a multi-coated inclusion of the form \eqnref{def:Omj} with $\Om_0$ given by 
$\Psi(w)= w+\frac{0.25}{w}$ with $\gamma=1$. The corresponding 
Grunsky matrix is diagonal (see \eqnref{ellipsegrunsky}), and thus, $\FF^{(1)}$ and $\FF^{(2)}$ are diagonal by Theorem \ref{FPTsystem}.

\smallskip
\smallskip
\begin{example}\label{example:ellipse}\rm
We construct semi-neutral inclusions with $L=1,2$. 
For $L=1$, we set $r_1=1.1$ and $(N_1,N_2) = (1,1)$; for $L=2$, we set $(r_1,r_2)=(1.1,\, 1.2)$ and $(N_1,N_2) = (2,2)$. 
The conductivities in the core and the background are set to be $\sigma_0=5$ and $\sigma_{L+1}=1$.
We set the initial guess for the iteration \eqnref{initialguess} as $\boldsymbol{\sigma}^{(0)}=(\sigma_1^{(0)},\dots,\sigma_L^{(0)})$ with $\sigma_j^{(0)} = 2^{(-1)^j} \quad \mbox{for } j=1,\dots,L$.
We then choose the conductivity values in the coating layers by following the numerical scheme in Subsection \ref{subsec:FPT_reducing}.
The resulting values are $\sigma_1=0.1216$ for the $1$-coated ellipse (i.e., $L=1$), and $(\sigma_1,\sigma_2)=(0.0782,\, 3.6782)$ for the $2$-coated ellipse (i.e., $L=2$), respectively. 

Figure \ref{ellipse_all} indicates the potential perturbation due to the uncoated (first column), $1$-coated (second column), and $2$-coated (third column) inclusions, where the background field is either $H(x)=x_1$ or $H(x)=x_2$. 
Colored curves represent level curves of the perturbed potential $u$, where we compute $u$ based on the boundary integral formulation for the conductivity transmission problem with the Nystr\"{o}m discretization; we refer the reader to \cite[Chapter 17]{Ammari:2013:MSM} for the details and computation codes.
Coated ellipses exhibit much smaller perturbations than the uncoated ellipse. 
Table \ref{table_ellipse_all} shows the first two diagonal elements of FPTs, and verifies that the obtained coated ellipses are semi-neutral inclusions with $L=1,2$.

\end{example}

\begin{figure}[H]
\begin{subfigure}{\linewidth}
\centering
\captionsetup{justification=centering}
\begin{minipage}{0.22\linewidth}
\subcaption*{Uncoated}
\end{minipage}\hspace*{5mm}
\begin{minipage}{0.22\linewidth}
\subcaption*{1-coated}
\end{minipage}\hspace*{5mm}
\begin{minipage}{0.22\linewidth}
\subcaption*{2-coated}
\end{minipage}\hspace*{5mm}
\end{subfigure}\\
\begin{subfigure}{\linewidth}
\centering
\begin{minipage}{0.22\linewidth}
\includegraphics[width=\linewidth, trim={37mm 21mm 32mm 17mm}, clip, frame]{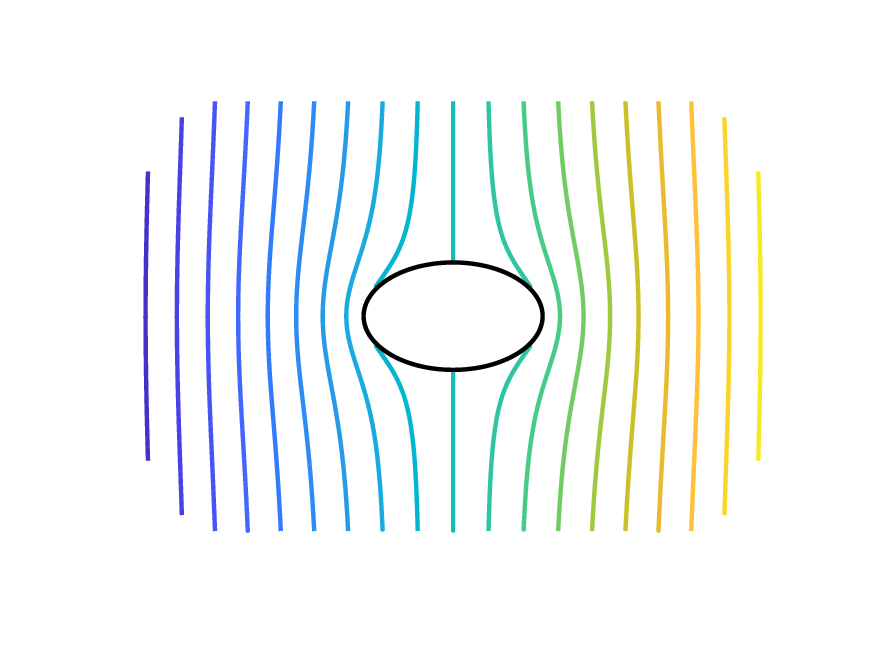}
\end{minipage}\hspace*{5mm}
\begin{minipage}{0.22\linewidth}
\includegraphics[width=\linewidth, trim={37mm 21mm 32mm 17mm}, clip, frame]{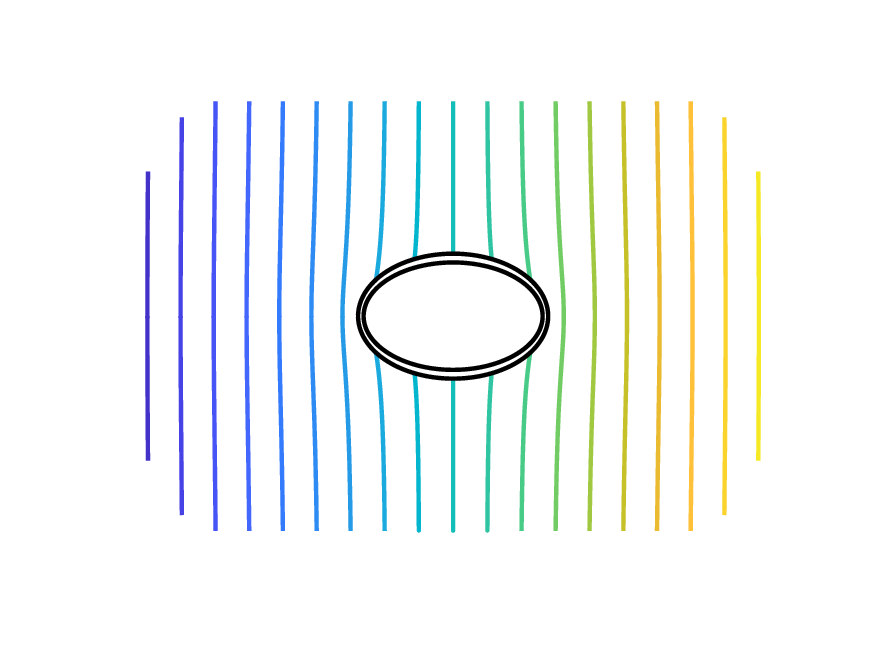}
\end{minipage}\hspace*{5mm}
\begin{minipage}{0.22\linewidth}
\includegraphics[width=\linewidth, trim={37mm 21mm 32mm 17mm}, clip, frame]{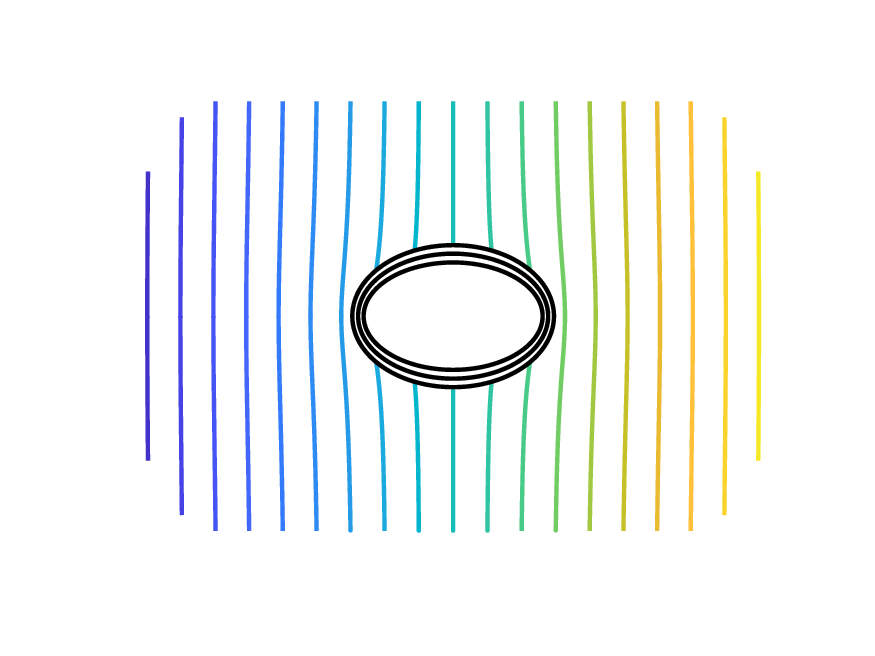}
\end{minipage}\hspace*{5mm}
\end{subfigure}\\[3mm]
\begin{subfigure}{\linewidth}
\centering
\begin{minipage}{0.22\linewidth}
\includegraphics[width=\linewidth, trim={37mm 21mm 32mm 17mm}, clip, frame]{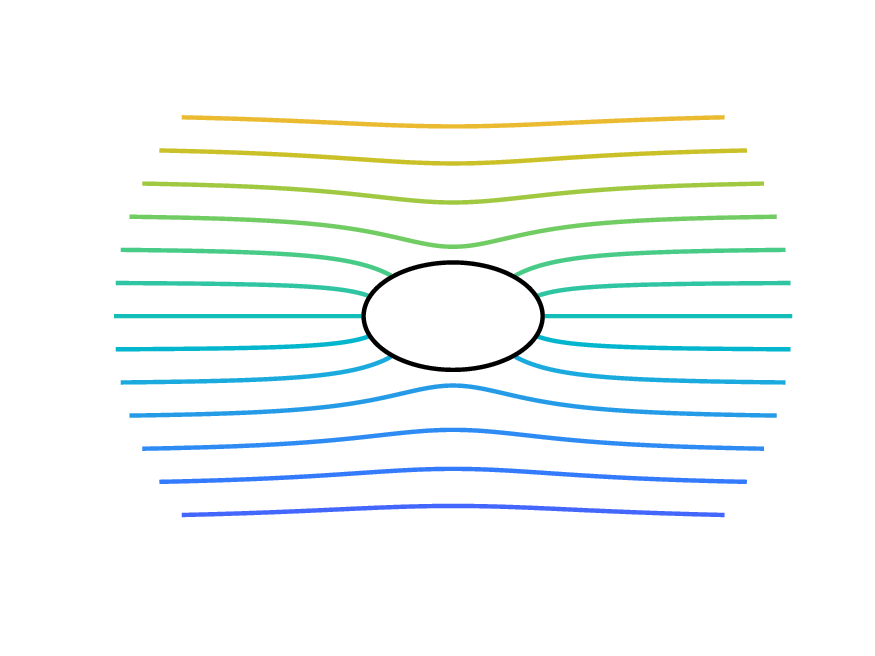}
\end{minipage}\hspace*{5mm}
\begin{minipage}{0.22\linewidth}
\includegraphics[width=\linewidth, trim={37mm 21mm 32mm 17mm}, clip, frame]{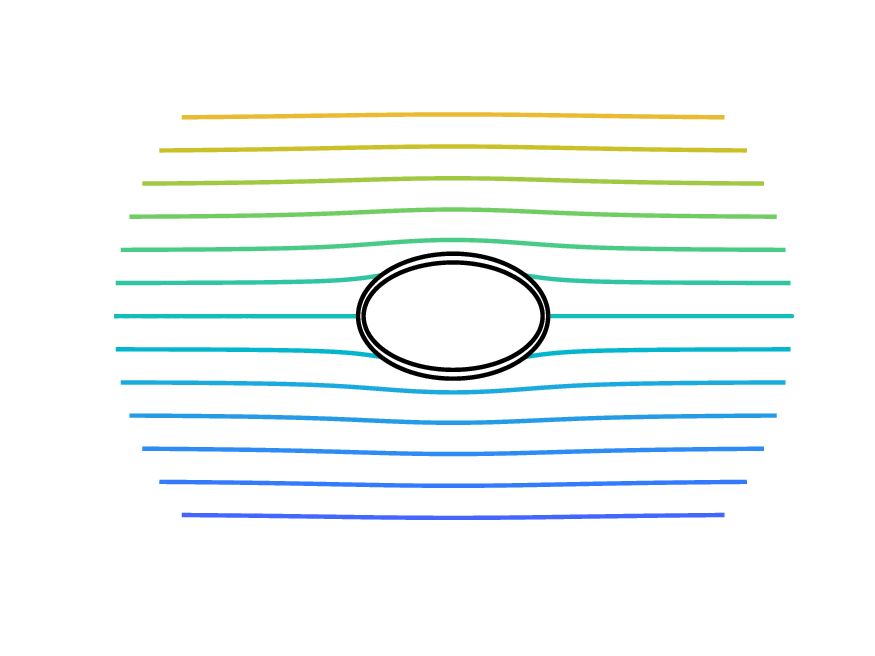}
\end{minipage}\hspace*{5mm}
\begin{minipage}{0.22\linewidth}
\includegraphics[width=\linewidth, trim={37mm 21mm 32mm 17mm}, clip, frame]{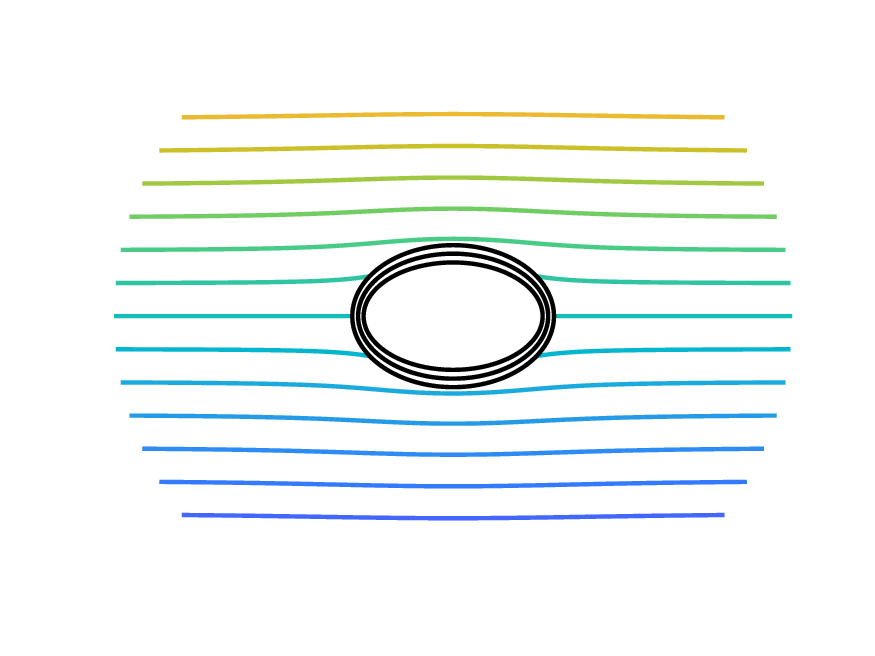}
\end{minipage}\hspace*{5mm}
\end{subfigure}
\caption{
Semi-neutral coated ellipses exhibit very small perturbations than the uncoated one.
}\label{ellipse_all}
\end{figure}
{\renewcommand{\arraystretch}{1.7}
\begin{table}[H]
\centering
\begin{tabular}{|c|c c|c c|}
\hline
& $\frac{\FF_{11}^{(1)}}{4\pi}$ & $\frac{\FF_{22}^{(1)}}{8\pi}$ & $\frac{\FF_{11}^{(2)}}{4\pi}$ & $\frac{\FF_{22}^{(2)}}{8\pi}$\\
\hline\hline
Uncoated & 0.1071 & 0.0277 & 0.6429 & 0.6652 \\
\hline
$1$--coated & 0.1866 & 0.0532 & $1.8803 \times10^{-15}$ & -0.3964 \\
\hline
$2$--coated & 0.2147& 0.0591 & $2.3761\times10^{-15}$ & $-3.5178\times10^{-15}$ \\
\hline
\end{tabular}
\caption{
The FPTs for the ellipse and coated ellipses in Example \ref{example:ellipse}. Since the core is an ellipse, $\FF_{mn}^{(1)}=\FF^{(2)}_{mn}=0$ for $m\neq n$. The values of the first FPTs, $\FF^{(1)}_{mn}$, of coated ellipses are similar to those corresponding to the core, i.e., the uncoated inclusion. The second FPTs, $\FF^{(2)}_{mn}$, vanish for $m,n\leq N$ with $N=1$ for the $1$-coated inclusion and $N=2$ for the $2$-coated inclusion.
}\label{table_ellipse_all}
\end{table}

\subsubsection{Non-elliptical shapes}
In this subsection, we provide two examples of non-elliptical shape. 
It is necessary to find the conductivity distribution that makes a coated inclusion semi-neutral. 
As Example 1, the uniform background field is given by $H(x)=x_1$ or $H(x)=x_2$.

\smallskip
\smallskip
\begin{example}\rm
Let us observe a kite-shaped domain with exterior conformal mapping,
$$\Psi(w) = w + \frac{0.1}{w} + \frac{0.25}{w^2} - \frac{0.05}{w^3} + \frac{0.05}{w^4} - \frac{0.04}{w^5} + \frac{0.02}{w^6}.$$ There is no rotational symmetry for this domain so that the vanishing property in Proposition \ref{rotationsym} does not hold. Figure \ref{kite} illustrates the potential perturbation $u$ on the background field $H$.
Table \ref{table_kite} shows the low-order FPTs. Each coated kite is semi-neutral because its second FPT is significantly smaller than the uncoated one.

In Figure \ref{kite}, the background field is given by $H(x)=x_1$. Colored curves represent contours of $u$. We commonly set $\gamma = 1$ and $\sigma_0 = 100$. We set the initial guess for the iteration \eqnref{initialguess} as $\boldsymbol{\sigma}^{(0)}=(\sigma_1^{(0)},\dots,\sigma_L^{(0)})$ with $\sigma_j^{(0)} =2^{(-1)^j}$ for $j=1,\dots,L$, except we put $\sigma_j^{(0)}  = 10^{(-1)^j}$ instead of $\sigma_j^{(0)}  = 2^{(-1)^j}$ for the $3$-coated kite in Figure \ref{kite}. 

For the $1$--coated kite, $r_1 = 1.1$, $\sigma_1=0.0964$, and $(N_1,N_2) = (1,1)$. For the $2$-coated kite, $(r_1,r_2) = (1.1,\,1.2)$, \ $(\sigma_1, \sigma_2)= (4.3533\times 10^{-6}, \,11.5212)$, and $(N_1,N_2) = (1,2)$. For the $3$-coated kite, $(r_1,r_2,r_3) = (1.1,\,1.2,\,1.3)$, \ $(\sigma_1, \sigma_2, \sigma_3)= (5.8354\times 10^{-6}, \,17.0950, \, 0.2358)$, and $(N_1,N_2) = (2,2)$. 
\end{example}

\begin{figure}[H]
\begin{subfigure}{\linewidth}
\centering
\captionsetup{justification=centering}
\begin{minipage}{0.22\linewidth}
\subcaption*{Uncoated}
\end{minipage}\hspace*{5mm}
\begin{minipage}{0.22\linewidth}
\subcaption*{1-coated}
\end{minipage}\hspace*{5mm}
\begin{minipage}{0.22\linewidth}
\subcaption*{2-coated}
\end{minipage}\hspace*{5mm}
\begin{minipage}{0.22\linewidth}
\subcaption*{3-coated}
\end{minipage}\hspace*{5mm}
\end{subfigure}\\
\begin{subfigure}{\linewidth}
\centering
\begin{minipage}{0.22\linewidth}
\includegraphics[width=\linewidth, trim={37mm 21mm 32mm 17mm}, clip, frame]{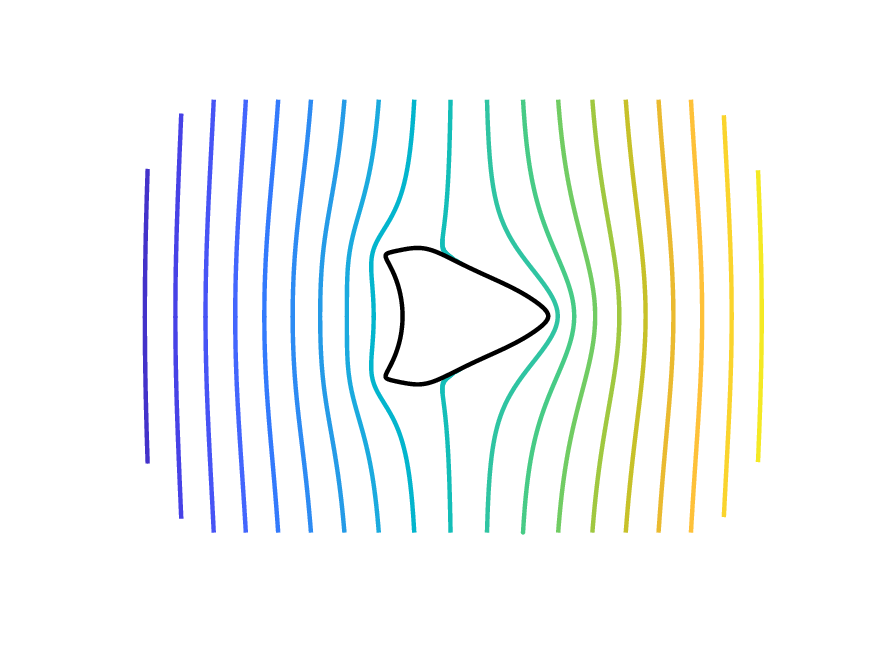}
\end{minipage}\hspace*{5mm}
\begin{minipage}{0.22\linewidth}
\includegraphics[width=\linewidth, trim={37mm 21mm 32mm 17mm}, clip, frame]{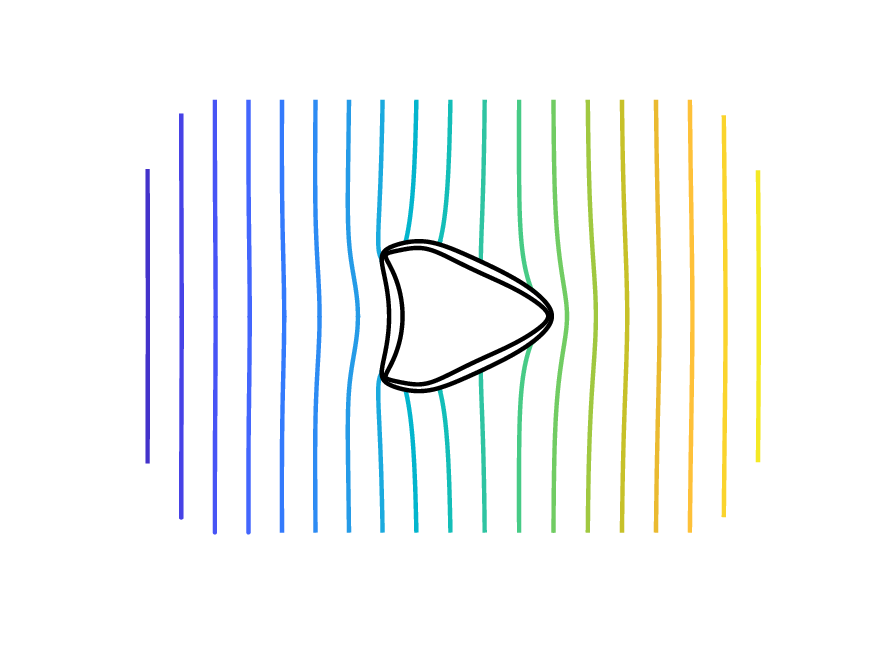}
\end{minipage}\hspace*{5mm}
\begin{minipage}{0.22\linewidth}
\includegraphics[width=\linewidth, trim={37mm 21mm 32mm 17mm}, clip, frame]{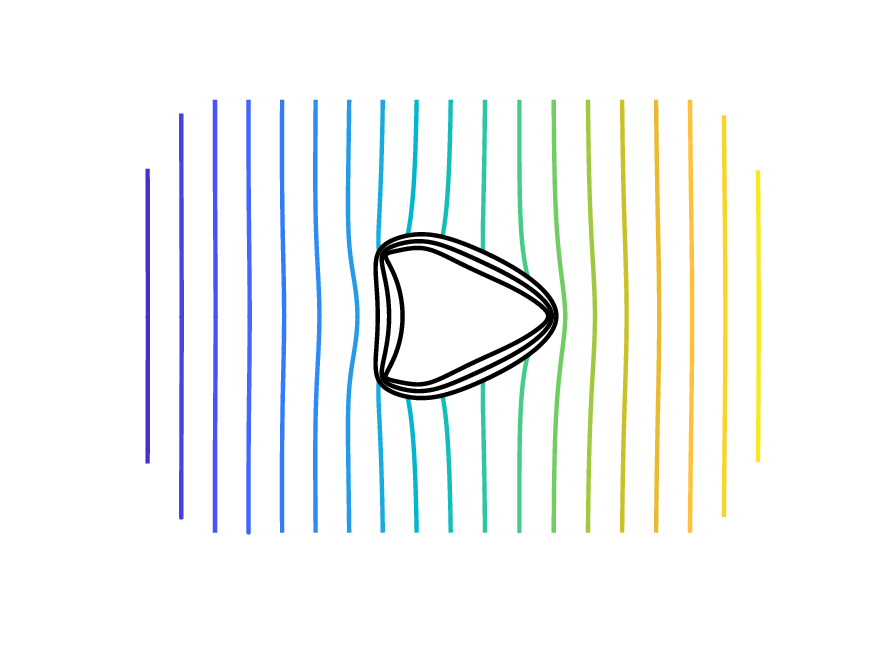}
\end{minipage}\hspace*{5mm}
\begin{minipage}{0.22\linewidth}
\includegraphics[width=\linewidth, trim={37mm 21mm 32mm 17mm}, clip, frame]{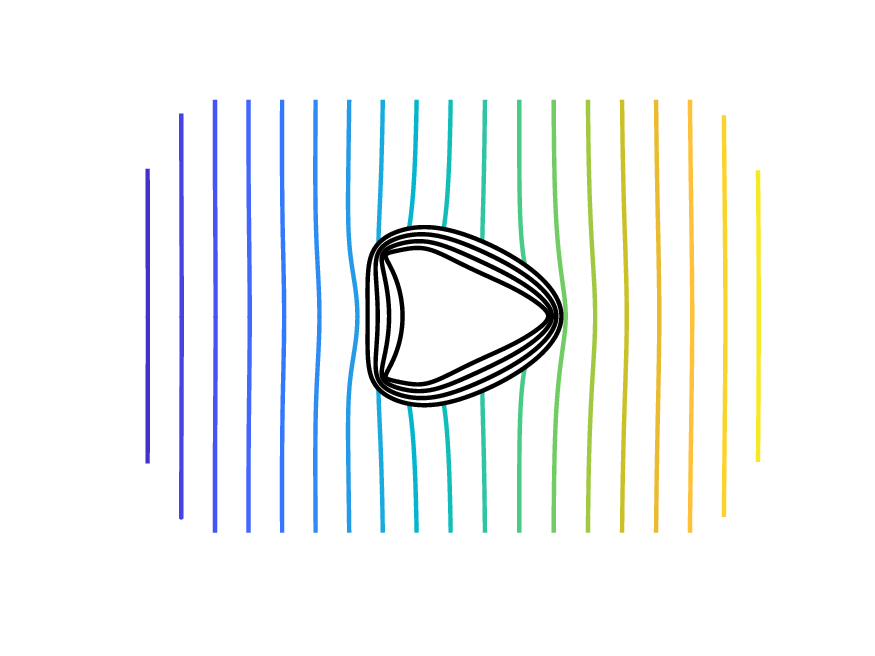}
\end{minipage}\hspace*{5mm}
\end{subfigure}\\[3mm]
\begin{subfigure}{\linewidth}
\centering
\begin{minipage}{0.22\linewidth}
\includegraphics[width=\linewidth, trim={37mm 21mm 32mm 17mm}, clip, frame]{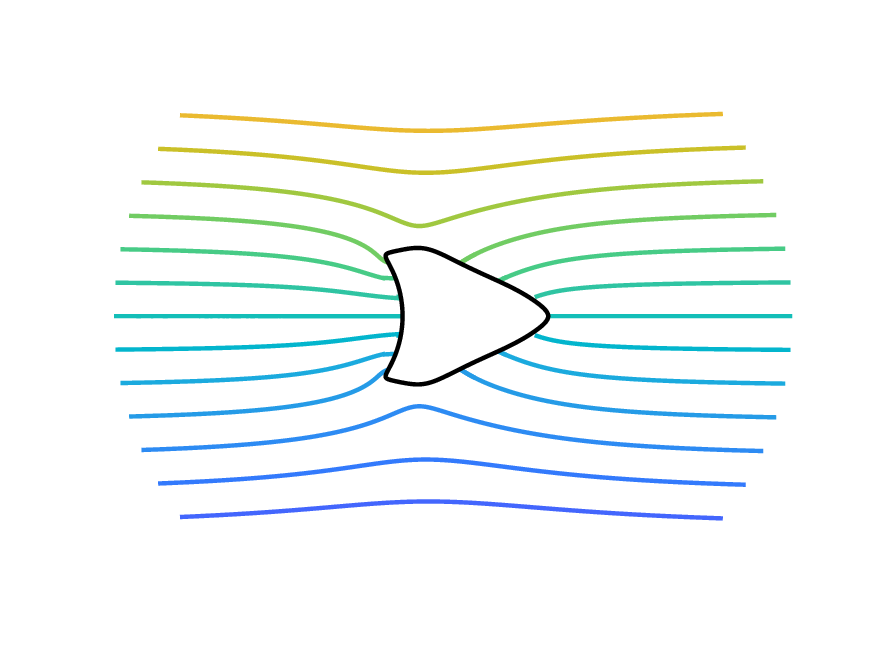}
\end{minipage}\hspace*{5mm}
\begin{minipage}{0.22\linewidth}
\includegraphics[width=\linewidth, trim={37mm 21mm 32mm 17mm}, clip, frame]{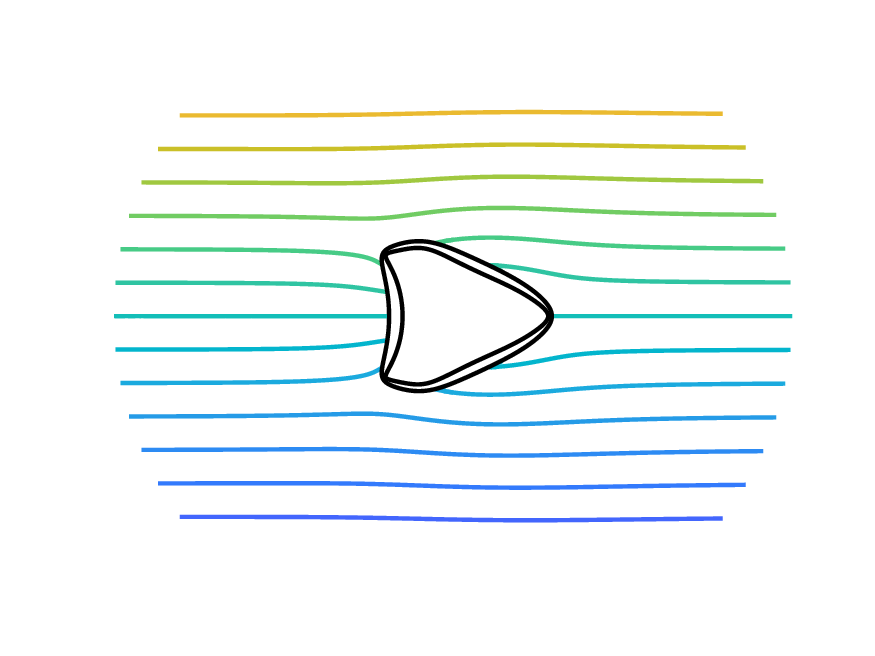}
\end{minipage}\hspace*{5mm}
\begin{minipage}{0.22\linewidth}
\includegraphics[width=\linewidth, trim={37mm 21mm 32mm 17mm}, clip, frame]{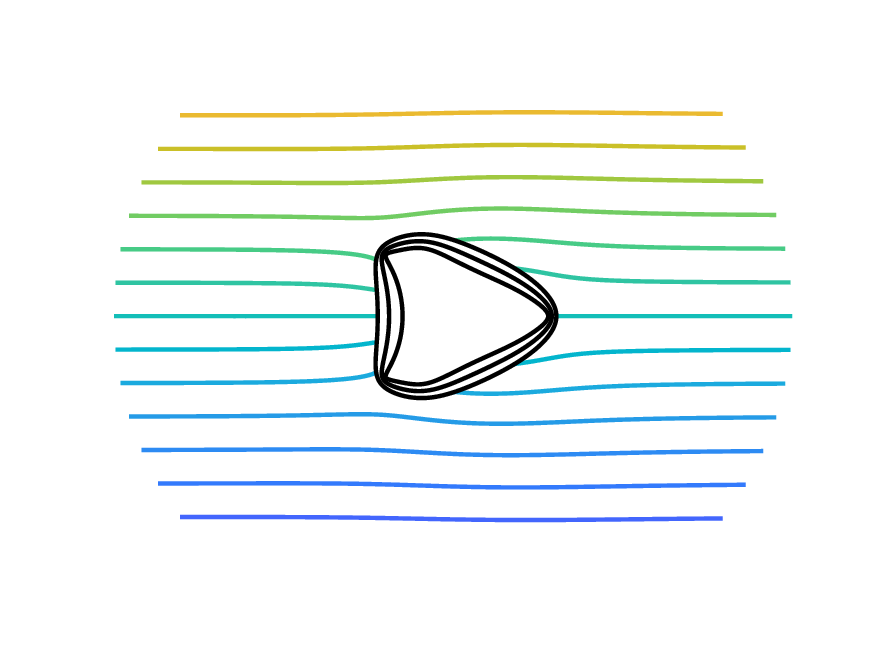}
\end{minipage}\hspace*{5mm}
\begin{minipage}{0.22\linewidth}
\includegraphics[width=\linewidth, trim={37mm 21mm 32mm 17mm}, clip, frame]{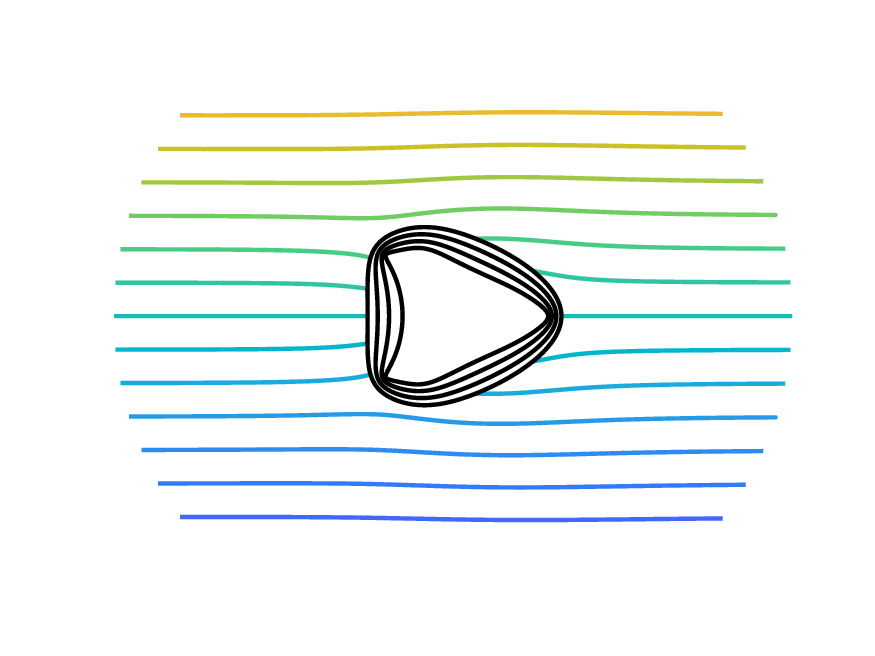}
\end{minipage}\hspace*{5mm}
\end{subfigure}
\caption{Semi-neutral coated kites show much smaller perturbations than the uncoated one.}\label{kite}
\end{figure}
{\renewcommand{\arraystretch}{1.7}
\begin{table}[H]
\centering
\begin{tabular}{|c|cc|cc|}
\hline
& $\frac{\FF_{11}^{(1)}}{4\pi}$ & $\frac{\FF_{12}^{(1)}}{8\pi}$ & $\frac{\FF_{11}^{(2)}}{4\pi}$ & $\frac{\FF_{12}^{(2)}}{8\pi}$\\
& $\frac{\FF_{21}^{(1)}}{4\pi}$ & $\frac{\FF_{22}^{(1)}}{8\pi}$ & $\frac{\FF_{21}^{(2)}}{4\pi}$ & $\frac{\FF_{22}^{(2)}}{8\pi}$\\
\hline\hline
\multirow{2}{*}{Uncoated} 
& 0.0956 & 0.2380	& 0.9726 & $7.6388 \times 10^{-4}$ \\
& 0.4760 & -0.0849 	& 0.0015 & 0.9720 \\
\hline
\multirow{2}{*}{$1$--coated} 
& 0.0987 & 0.2473	& $1.9410 \times 10^{-15}$ & $1.7713 \times 10^{-4}$ \\
& 0.4946 & -0.0891	& $3.5426 \times 10^{-4}$ & -0.4777 \\
\hline
\multirow{2}{*}{$2$--coated} 
& 0.1000 & 0.2500	& $3.8723 \times 10^{-13}$ & $2.3846 \times 10^{-13}$ \\
& 0.5000 & -0.0900 	& $4.7696 \times 10^{-13}$ & 0.6842 \\
\hline
\multirow{2}{*}{$3$--coated} 
& 0.1000 & 0.2500	& $-2.3285 \times 10^{-7}$ & $1.3125 \times 10^{-14}$ \\
& 0.5000 & -0.0900 	& $2.6250 \times 10^{-14}$ & $9.4847 \times 10^{-13}$ \\
\hline
\end{tabular}
\caption{Low-order FPTs of the uncoated and coated kites in Figure \ref{kite}. As we coat the inclusion, the first FPT is almost invariant, but the second FPTs decrease significantly.}\label{table_kite}
\end{table}
}

\smallskip
\smallskip
\begin{example}\rm
In this example, we assume that the core $\Om_0$ is given by the conformal mapping $$\Psi(w) = w + \frac{0.2}{w^4}.$$ Since the resulting star-shape domain has rotational symmetry of order $5$, the associated FPTs show the following periodicity from Proposition \ref{rotationsym}: 
\begin{align*}
&\FF_{mn}^{(1)} = 0 \quad \mbox{for } m+n \not\equiv 0 \ (\mbox{mod } 5),\\
&\FF_{mn}^{(2)} = 0 \quad \mbox{for } m-n \not\equiv 0 \ (\mbox{mod } 5).
\end{align*}
Then, $\FF_{mn}^{(1)}=0$ for all $m,n\le 2$, and $\FF_{mn}^{(2)}=0$ for all $m\neq n$ with $m,n \le 5$. Hence, we focus only on vanishing the nonzero leading terms of FPTs.
Figure \ref{star} that illustrates contours of the potential perturbation for a given background field $H$. Table \ref{table_star} compares the leading FPTs of domains in Figure \ref{star}.

In Figure \ref{star}, the background field is given by $H(x)=x_2$. Colored curves represent contours of $u$. We commonly set $\gamma = 1$ and $\sigma_0 = 0.25$. The initial guess for the iteration \eqnref{initialguess} is given by $\boldsymbol{\sigma}^{(0)}=(\sigma_1^{(0)},\dots,\sigma_L^{(0)})$ with $\sigma_j^{(0)} = 10^{(-1)^{j+1}} \quad \mbox{for } j=1,\dots,L$. 

For the $1$--coated star, $r_1 = 1.1$, $\sigma_1= 7.0224$, and $(N_1,N_2)=(1,1)$. For the $2$-coated star, $(r_1,r_2) = (1.1,\,1.2)$, $(\sigma_1, \sigma_2)= (11.4177, \,0.2763)$, and $(N_1,N_2)=(2,2)$. For the $3$-coated star, $(r_1,r_2,r_3) = (1.1,\,1.2,\,1.3)$, $(\sigma_1, \sigma_2, \sigma_3)= (540.1332,\, 0.0595,\, 4.2144)$, and $(N_1,N_2)=(2,6)$

\end{example}

\begin{figure}[H]
\begin{subfigure}{\linewidth}
\centering
\captionsetup{justification=centering}
\begin{minipage}{0.22\linewidth}
\subcaption*{Uncoated}
\end{minipage}\hspace*{5mm}
\begin{minipage}{0.22\linewidth}
\subcaption*{1-coated}
\end{minipage}\hspace*{5mm}
\begin{minipage}{0.22\linewidth}
\subcaption*{2-coated}
\end{minipage}\hspace*{5mm}
\begin{minipage}{0.22\linewidth}
\subcaption*{3-coated}
\end{minipage}\hspace*{5mm}
\end{subfigure}\\
\begin{subfigure}{\linewidth}
\centering
\begin{minipage}{0.22\linewidth}
\includegraphics[width=\linewidth, trim={37mm 21mm 32mm 17mm}, clip, frame]{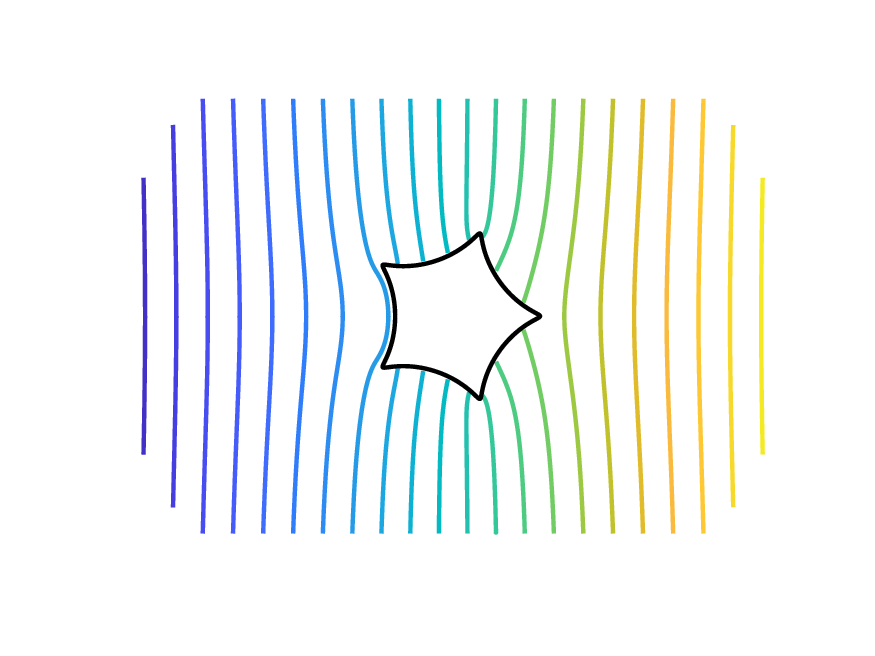}
\end{minipage}\hspace*{5mm}
\begin{minipage}{0.22\linewidth}
\includegraphics[width=\linewidth, trim={37mm 21mm 32mm 17mm}, clip, frame]{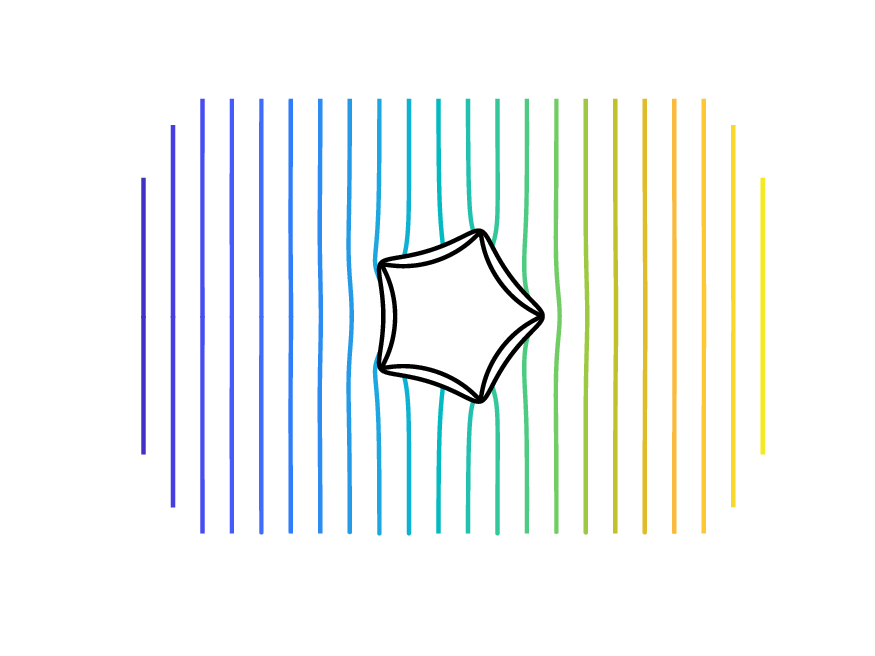}
\end{minipage}\hspace*{5mm}
\begin{minipage}{0.22\linewidth}
\includegraphics[width=\linewidth, trim={37mm 21mm 32mm 17mm}, clip, frame]{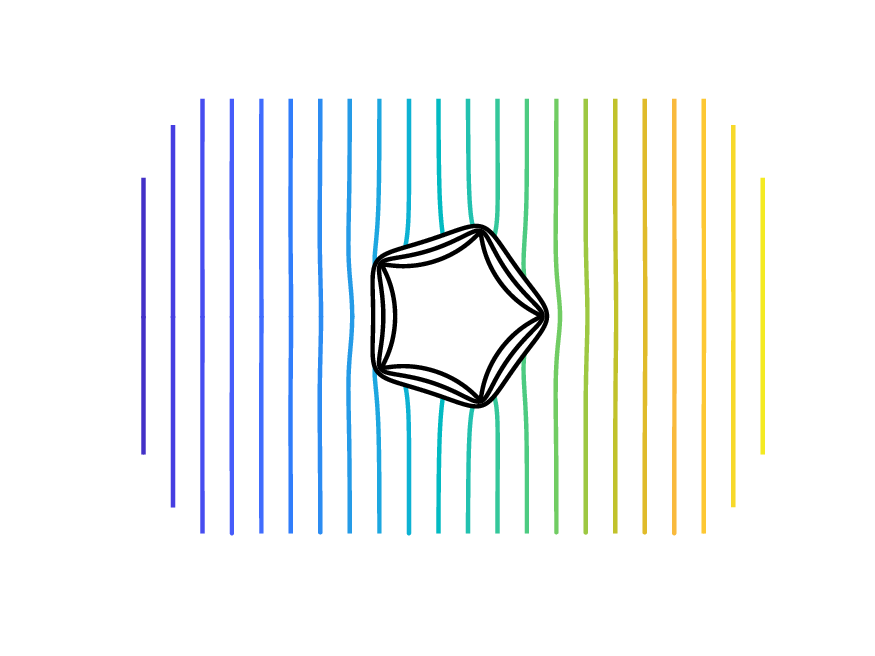}
\end{minipage}\hspace*{5mm}
\begin{minipage}{0.22\linewidth}
\includegraphics[width=\linewidth, trim={37mm 21mm 32mm 17mm}, clip, frame]{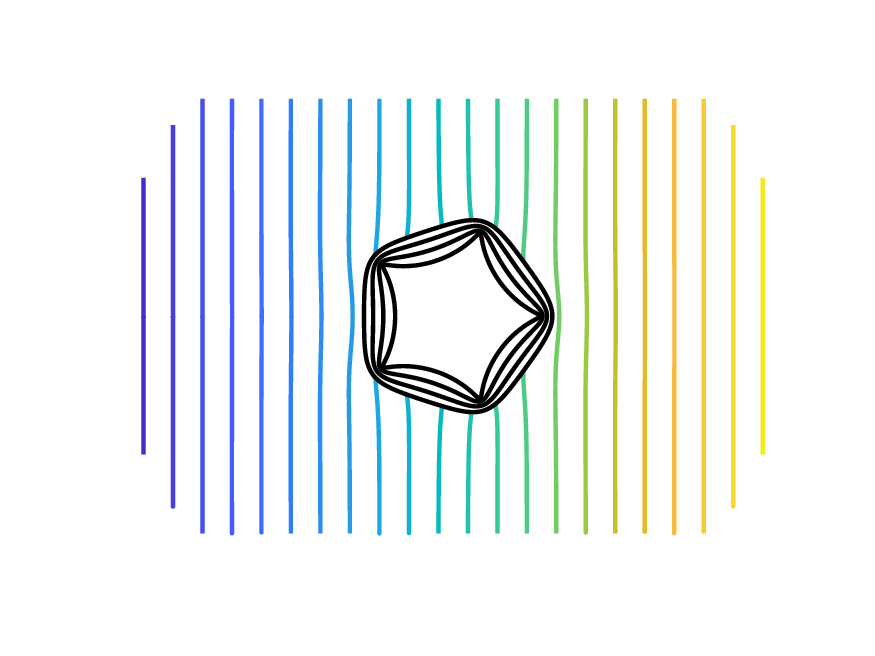}
\end{minipage}\hspace*{5mm}
\end{subfigure}\\[3mm]
\begin{subfigure}{\linewidth}
\centering
\begin{minipage}{0.22\linewidth}
\includegraphics[width=\linewidth, trim={37mm 21mm 32mm 17mm}, clip, frame]{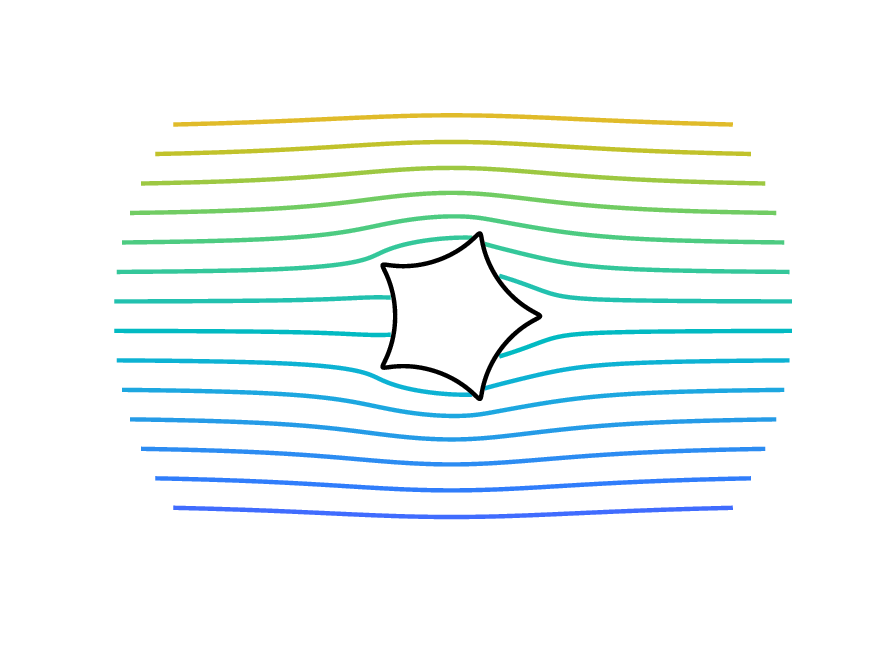}
\end{minipage}\hspace*{5mm}
\begin{minipage}{0.22\linewidth}
\includegraphics[width=\linewidth, trim={37mm 21mm 32mm 17mm}, clip, frame]{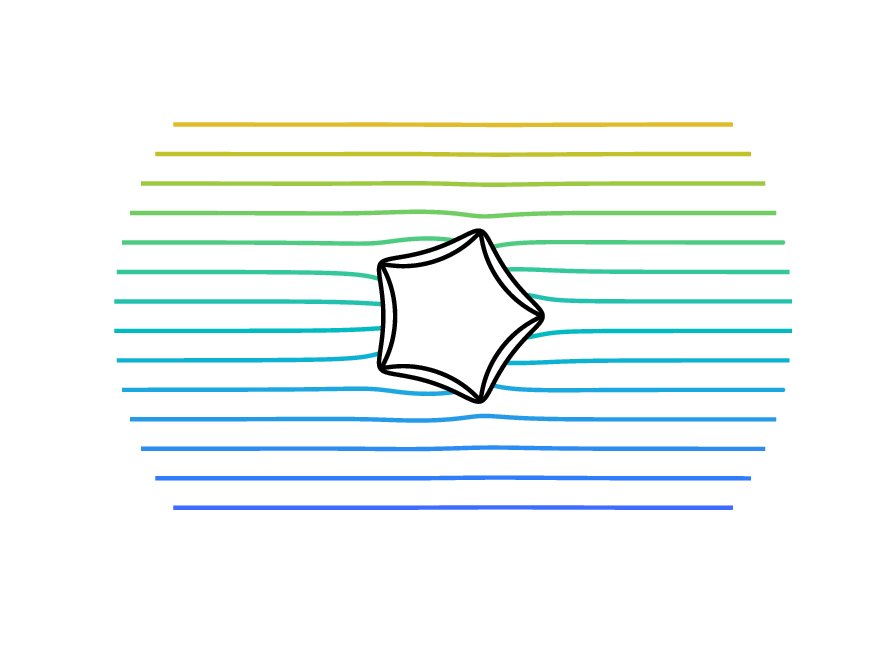}
\end{minipage}\hspace*{5mm}
\begin{minipage}{0.22\linewidth}
\includegraphics[width=\linewidth, trim={37mm 21mm 32mm 17mm}, clip, frame]{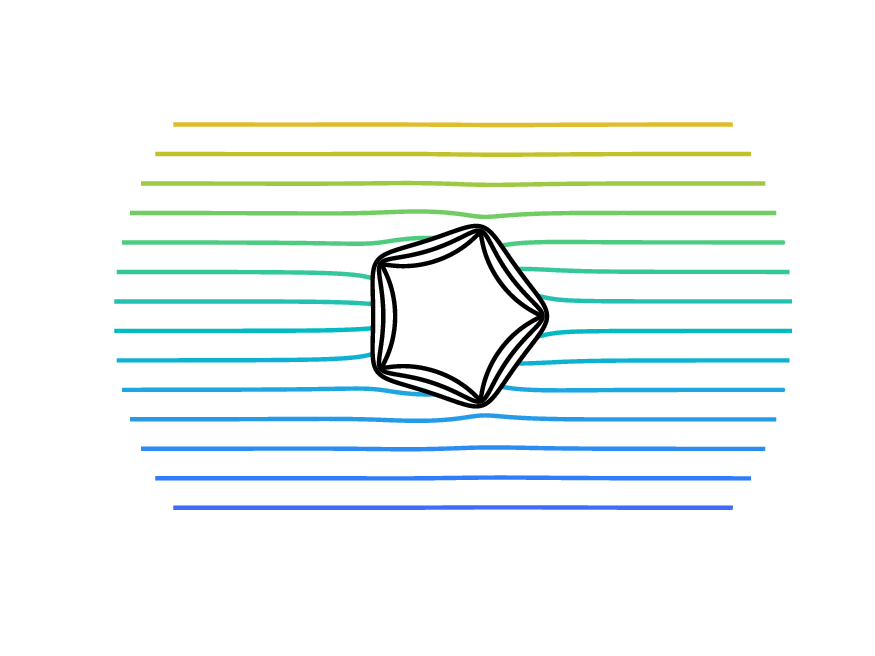}
\end{minipage}\hspace*{5mm}
\begin{minipage}{0.22\linewidth}
\includegraphics[width=\linewidth, trim={37mm 21mm 32mm 17mm}, clip, frame]{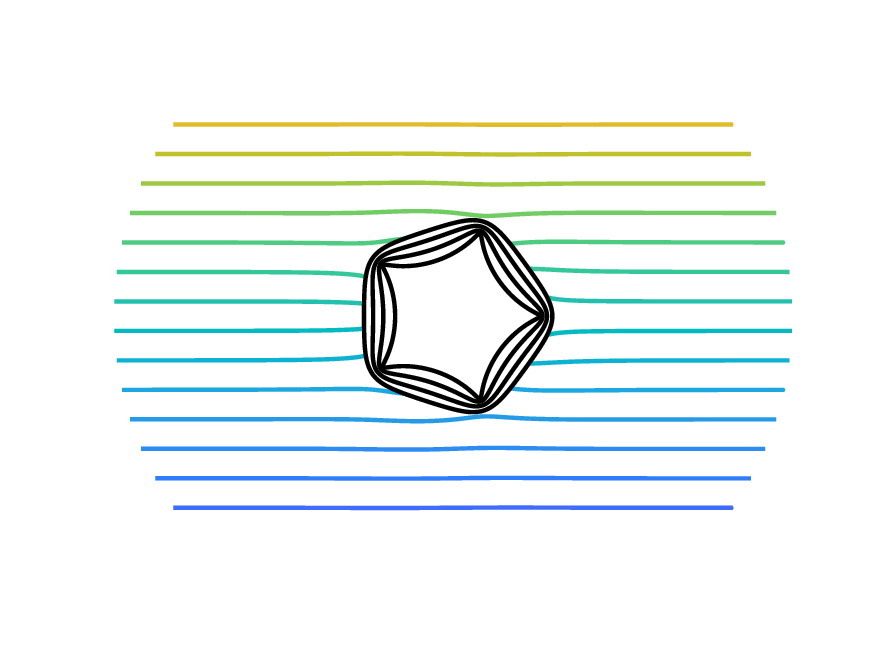}
\end{minipage}\hspace*{5mm}
\end{subfigure}
\caption{Semi-neutral coated stars have quite smaller perturbations than the uncoated star.}\label{star}
\end{figure}
{\renewcommand{\arraystretch}{1.7}
\begin{table}[H]
\centering
\begin{tabular}{|c|cc|ccc|}
\hline
& $\frac{\FF_{14}^{(1)}}{16\pi}$ & $\frac{\FF_{23}^{(1)}}{12\pi}$ & $\frac{\FF_{11}^{(2)}}{4\pi}$ & $\frac{\FF_{16}^{(2)}}{24\pi}$ & 	$\frac{\FF_{22}^{(2)}}{8\pi}$\\
\hline\hline
Uncoated & 0.0593 & 0.1142
& -0.5325 & 0.0206 & -0.4941 \\
\hline
$1$--coated & 0.1507 & 0.3083
& $1.4624\times 10^{-15}$ & 0.0092 & 0.3766 \\
\hline
$2$--coated & 0.1823 & 0.3699
& $-5.5394\times 10^{-15}$ & 0.0030 & $-1.5689\times 10^{-14}$ \\
\hline
$3$--coated & 0.2000 & 0.4000
& $-7.5966\times 10^{-13}$ & $2.6694\times 10^{-7}$ & $2.3356\times 10^{-13}$ \\
\hline
\end{tabular}
\caption{The leading terms of the FPTs for the uncoated and coated stars in Figure \ref{star}. As we coat the inclusion, the second FPTs vanish.}\label{table_star}
\end{table}
}

\bibliography{2020_GME_JMPA_ref}{}

\begin{thebibliography}{10}

\bibitem{Ahlfors:1963:QR}
Lars~V. Ahlfors.
\newblock Quasiconformal reflections.
\newblock {\em Acta Math.}, 109:291--301, 1963.

\bibitem{Alu:2005:ATP}
A.~Al{\`u} and N.~Engheta.
\newblock Achieving transparency with plasmonic and metamaterial coatings.
\newblock {\em Phys. Rev. - Stat. Nonlinear Soft Matter Phys.}, 72(1), 2005.

\bibitem{Ammari:2018:MNF}
H.~Ammari, D.S. Choi, and S.~Yu.
\newblock A mathematical and numerical framework for near-field optics.
\newblock {\em Proc. R. Soc. A: Math. Phys. Eng. Sci.}, 474(2217), 2018.

\bibitem{Ammari:2013:MSM}
H.~Ammari, J.~Garnier, W.~Jing, H.~Kang, M.~Lim, K.~S{\o}lna, and H.~Wang.
\newblock {\em {Mathematical and Statistical Methods for Multistatic Imaging}},
  volume 2098 of {\em Lecture Notes in Mathematics}.
\newblock Springer International Publishing, 2013.

\bibitem{Ammari:2007:PMT}
H.~Ammari and H.~Kang.
\newblock {\em {Polarization and Moment Tensors With Applications to Inverse
  Problems and Effective Medium Theory}}, volume 162 of {\em Applied
  Mathematical Sciences}.
\newblock Springer-Verlag New York, 2007.

\bibitem{Ammari:2013:ENC1}
H.~Ammari, H.~Kang, H.~Lee, and M.~Lim.
\newblock {Enhancement of near-cloaking using generalized polarization tensors
  vanishing structures. Part I: The conductivity problem}.
\newblock {\em Commun. Math. Phys.}, 317(1):253--266, 2013.

\bibitem{Ammari:2013:ENCm}
H.~Ammari, H.~Kang, H.~Lee, M.~Lim, and S.~Yu.
\newblock {Enhancement of near cloaking for the full Maxwell equations}.
\newblock {\em SIAM J. Appl. Math.}, 73(6):2055--2076, 2013.

\bibitem{Ando:2016:APR}
K.~Ando and H.~Kang.
\newblock {Analysis of plasmon resonance on smooth domains using spectral
  properties of the Neumann--Poincar{\'e} operator}.
\newblock {\em J. Math. Anal. Appl.}, 435(1):162--178, 2016.

\bibitem{Beckermann:2018:BOP}
B.~Beckermann and N.~Stylianopoulos.
\newblock Bergman orthogonal polynomials and the {G}runsky matrix.
\newblock {\em Constr. Approx.}, 47(2):211--235, 2018.

\bibitem{Bonnetier:2012:PBG}
{\'E}.~Bonnetier and F.~Triki.
\newblock {Pointwise bounds on the gradient and the spectrum of the
  Neumann--Poincar{\'e} operator: The case of 2 discs}.
\newblock {\em Contemp. Math.}, 577:81--92, 2012.

\bibitem{Caratheodory:1913:GBR}
C.~Carath\'{e}odory.
\newblock {\"{U}ber die gegenseitige Beziehung der R\"{a}nder bei der konformen
  Abbildung des Inneren einer Jordanschen Kurve auf einen Kreis}.
\newblock {\em Math. Ann.}, 73(2):305--320, 1913.

\bibitem{Cherkaev:2021:GSE}
E.~Cherkaev, M.~Kim, and M.~Lim.
\newblock {Geometric series expansion of the Neumann--Poincar{\'e} operator:
  application to composite materials}.
\newblock {\em To appear in Eur. J. Appl. Math.}, 2021.

\bibitem{Choi:2020:ASR}
D.~Choi, J.~Kim, and M.~Lim.
\newblock Analytical shape recovery of a conductivity inclusion based on faber
  polynomials.
\newblock {\em Math. Ann.}, 2020.

\bibitem{Choi:2021:EEC}
D.~Choi, K.~Kim, and M.~Lim.
\newblock An extension of the eshelby conjecture to domains of general shape in
  anti-plane elasticity.
\newblock {\em J. Math. Anal. Appl.}, 495(2), 2021.

\bibitem{Choi:2018:CEP}
D.S. Choi, J.~Helsing, and M.~Lim.
\newblock Corner effects on the perturbation of an electric potential.
\newblock {\em SIAM J. Appl. Math.}, 78(3):1577--1601, 2018.

\bibitem{Duren:1983:UF}
P.~L. Duren.
\newblock {\em {Univalent Functions}}, volume 259 of {\em Grundlehren der
  Mathematischen Wissenschaften}.
\newblock Springer-Verlag New York, 1983.

\bibitem{Escauriaza:1992:RTW}
L.~Escauriaza, E.B. Fabes, and G.~Verchota.
\newblock {On a regularity theorem for weak solutions to transmission problems
  with internal Lipschitz boundaries}.
\newblock {\em Proc. Am. Math. Soc.}, 115(4):1069--1076, 1992.

\bibitem{Escauriaza:1993:RPS}
L.~Escauriaza and J.K. Seo.
\newblock Regularity properties of solutions to transmission problems.
\newblock {\em Trans. Am. Math. Soc.}, 338(1):405--430, 1993.

\bibitem{Faber:1903:PE}
G.~Faber.
\newblock {{\"U}ber polynomische Entwicklungen}.
\newblock {\em Math. Ann.}, 57(3):389--408, 1903.

\bibitem{Feng:2017:CGV}
T.~Feng, H.~Kang, and H.~Lee.
\newblock {Construction of GPT-vanishing structures using shape derivative}.
\newblock {\em J. Comput. Math.}, 35(5):569--585, 2017.

\bibitem{Grabovsky:1995:MME1}
Y.~Grabovsky and R.V. Kohn.
\newblock {Microstructures minimizing the energy of a two phase elastic
  composite in two space dimensions. I: The confocal ellipse construction}.
\newblock {\em J. Mech. Phys. Solids}, 43(6):933--947, 1995.

\bibitem{Grunsky:1939:KSA}
H.~Grunsky.
\newblock {Koeffizientenbedingungen f\"{u}r schlicht abbildende meromorphe
  Funktionen}.
\newblock {\em Math. Z.}, 45(1):29--61, 1939.

\bibitem{Hashin:1962:EMH}
Z.~Hashin.
\newblock The elastic moduli of heterogeneous materials.
\newblock {\em J. Appl. Mech.}, 29(1):143--150, 1960.

\bibitem{Hashin:1985:LIE}
Z.~Hashin.
\newblock Large isotropic elastic deformation of composites and porous media.
\newblock {\em Int. J. Solids Struct.}, 21(7):711--720, 1985.

\bibitem{Hashin:1962:VAT}
Z.~Hashin and S.~Shtrikman.
\newblock A variational approach to the theory of the effective magnetic
  permeability of multiphase materials.
\newblock {\em J. Appl. Phys.}, 33(10):3125--3131, 1962.

\bibitem{Helsing:2013:SIE}
J.~Helsing.
\newblock {Solving integral equations on piecewise smooth boundaries using the
  RCIP method: A tutorial}.
\newblock {\em Abstr. Appl. Anal.}, 2013, 2013.

\bibitem{Helsing:2017:CSN}
J.~Helsing, H.~Kang, and M.~Lim.
\newblock {Classification of spectra of the Neumann--Poincar{\'e} operator on
  planar domains with corners by resonance}.
\newblock {\em Ann. Inst. Henri Poincare (C) Anal. Non Lineaire},
  34(4):991--1011, 2017.

\bibitem{Jimenez:2013:NNI}
S.~Jim{\'e}nez, B.~Vernescu, and W.~Sanguinet.
\newblock {Nonlinear neutral inclusions: Assemblages of spheres}.
\newblock {\em Int. J. Solids Struct.}, 50(14-15):2231--2238, 2013.

\bibitem{Johnston:1987:FER}
E.~H. Johnston.
\newblock Faber expansions of rational and entire functions.
\newblock {\em SIAM J. Math. Anal.}, 18(5):1235--1247, 1987.

\bibitem{Jung:2020:DEE}
Y.~Jung and M.~Lim.
\newblock {A decay estimate for the eigenvalues of the Neumann--Poincar\'{e}
  operator using the Grunsky coefficients}.
\newblock {\em Proc. Am. Math. Soc.}, 148(2):591--600, 2020.

\bibitem{Jung:2021:SEL}
Younghoon Jung and Mikyoung Lim.
\newblock Series expansions of the layer potential operators using the {F}aber
  polynomials and their applications to the transmission problem.
\newblock {\em SIAM J. Math. Anal.}, 53(2):1630--1669, 2021.

\bibitem{Kang:2015:LPA}
H.~Kang.
\newblock Layer potential approaches to interface problems.
\newblock In {\em Inverse problems and imaging}, volume~44 of {\em Panoramas et
  Synth{\`e}ses}, pages 63--110. Societe Mathematique de France, Paris, 2014.

\bibitem{Kang:2014:CIF}
H.~Kang and H.~Lee.
\newblock Coated inclusions of finite conductivity neutral to multiple fields
  in two-dimensional conductivity or anti-plane elasticity.
\newblock {\em Eur. J. Appl. Math.}, 25(3):329--338, 2014.

\bibitem{Kang:2016:OBV}
H.~Kang, H.~Lee, and S.~Sakaguchi.
\newblock An over-determined boundary value problem arising from neutrally
  coated inclusions in three dimensions.
\newblock {\em Ann. Sc. Norm. Super. Pisa - Cl. Sci.}, 16(4):1193--1208, 2016.

\bibitem{Kang:2019:CWN}
H.~Kang and X.~Li.
\newblock Construction of weakly neutral inclusions of general shape by
  imperfect interfaces.
\newblock {\em SIAM J. Appl. Math.}, 79(1):396--414, 2019.

\bibitem{Kellogg:1929:FPT}
O.D. Kellogg.
\newblock {\em {Foundations of Potential Theory}}, volume~31 of {\em Die
  Grundlehren der Mathematischen Wissenschaften}.
\newblock Springer-Verlag Berlin Heidelberg, 1929.

\bibitem{Kerker:1975:IB}
M.~Kerker.
\newblock Invisible bodies.
\newblock {\em J. Opt. Soc. Am.}, 65(4):376--379, 1975.

\bibitem{Khavinson:2007:PVP}
D.~Khavinson, M.~Putinar, and H.S. Shapiro.
\newblock Poincar\'{e}'s variational problem in potential theory.
\newblock {\em Arch. Ration. Mech. Anal.}, 185(1):143--184, 2007.

\bibitem{Kuehnau:1971:VKG}
R.~K\"{u}hnau.
\newblock Verzerrungss\"{a}tze und {K}oeffizientenbedingungen vom
  {G}runskyschen {T}yp f\"{u}r quasikonforme {A}bbildungen.
\newblock {\em Math. Nachr.}, 48:77--105, 1971.

\bibitem{Landy:2013:FPU}
N.~Landy and D.R. Smith.
\newblock A full-parameter unidirectional metamaterial cloak for microwaves.
\newblock {\em Nat. Mater.}, 12(1):25--28, 2013.

\bibitem{Lim:2001:SBI}
M.~Lim.
\newblock Symmetry of a boundary integral operator and a characterization of a
  ball.
\newblock {\em Ill. J. Math.}, 45(2):537--543, 2001.

\bibitem{Milton:2002:TC}
G.W. Milton.
\newblock {\em {The Theory of Composites}}, volume~6 of {\em Cambridge
  Monographs on Applied and Computational Mathematics}.
\newblock Cambridge University Press, Cambridge, 2002.

\bibitem{Milton:2001:NCI}
G.W. Milton and S.K. Serkov.
\newblock Neutral coated inclusions in conductivity and anti-plane elasticity.
\newblock {\em Proc. R. Soc. A: Math. Phys. Eng. Sci.}, 457(2012):1973--1997,
  2001.

\bibitem{Polya:1951:IIM}
G.~P{\'o}lya and G.~Szeg{\H{o}}.
\newblock {\em {Isoperimetric Inequalities in Mathematical Physics. (AM-27)}},
  volume~27 of {\em Annals of Mathematics Studies}.
\newblock Princeton University Press, 1951.

\bibitem{Pommerenke:1975:UF}
C.~Pommerenke.
\newblock {\em Univalent functions}.
\newblock Vandenhoeck \& Ruprecht, G\"{o}ttingen, 1975.

\bibitem{Pommerenke:1992:BBC}
C.~Pommerenke.
\newblock {\em {Boundary Behaviour of Conformal Maps}}, volume 299 of {\em
  Grundlehren der Mathematischen Wissenschaften}.
\newblock Springer-Verlag Berlin Heidelberg, 1992.

\bibitem{Sihvola:1999:EMF}
A.~Sihvola.
\newblock {\em {Electromagnetic Mixing Formulas and Applications}}, volume~47
  of {\em Electromagnetic Waves}.
\newblock Institution of Engineering and Technology, 1999.

\bibitem{Sihvola:1997:DPI}
A.H. Sihvola.
\newblock On the dielectric problem of isotrophic sphere in anisotropic medium.
\newblock {\em Electromagnetics}, 17(1):69--74, 1997.

\bibitem{Springer:1964:FEQ}
G.~Springer.
\newblock Fredholm eigenvalues and quasiconformal mapping.
\newblock {\em Acta Math.}, 111:121--142, 1964.

\bibitem{Suetin:1998:SFP}
P.~K. Suetin.
\newblock {\em Series of {F}aber polynomials}, volume~1 of {\em Analytical
  Methods and Special Functions}.
\newblock Gordon and Breach Science Publishers, 1998.

\bibitem{Verchota:1984:LPR}
G.~Verchota.
\newblock {Layer potentials and regularity for the Dirichlet problem for
  Laplace's equation in Lipschitz domains}.
\newblock {\em J. Funct. Anal.}, 59(3):572--611, 1984.

\bibitem{Wang:2013:MNC}
X.~Wang and P.~Schiavone.
\newblock A neutral multi-coated sphere under non-uniform electric field in
  conductivity.
\newblock {\em Z. Angew. Math. Phys.}, 64(3):895--903, 2013.

\bibitem{Zhou:2006:DEW}
X.~Zhou and G.~Hu.
\newblock Design for electromagnetic wave transparency with metamaterials.
\newblock {\em Phys. Rev. - Stat. Nonlinear Soft Matter Phys.}, 74(2), 2006.

\bibitem{Zhou:2007:AWT}
X.~Zhou and G.~Hu.
\newblock Acoustic wave transparency for a multilayered sphere with acoustic
  metamaterials.
\newblock {\em Phys. Rev. - Stat. Nonlinear Soft Matter Phys.}, 75(4), 2007.

\bibitem{Zhou:2008:EWT}
X.~Zhou, G.~Hu, and T.~Lu.
\newblock Elastic wave transparency of a solid sphere coated with
  metamaterials.
\newblock {\em Phys. Rev. B - Condens. Matter Mater. Phys.}, 77(2), 2008.

\end{thebibliography}
\bibliographystyle{plain}

\end{document}